\documentclass[11pt]{article}
\usepackage{graphicx}
\usepackage{amsmath}
\usepackage{mathrsfs}
\usepackage{amssymb}
\usepackage{amsthm}
\usepackage{cite}
\usepackage{color}
\usepackage[colorlinks=true,linkcolor=blue!60!black,citecolor=teal!80!black,urlcolor=RoyalBlue]{hyperref}
\usepackage[utf8x]{inputenc} 
\usepackage{fixmath}
\usepackage[usenames,dvipsnames]{xcolor}
\usepackage[left=2.5cm,right=2.5cm,top=2cm,bottom=2cm]{geometry}
\usepackage{bm}
\usepackage{tikz-cd}
\numberwithin{equation}{section}
\usepackage{chngcntr}
\counterwithout{footnote}{section}

\usepackage{thmtools}
\usepackage{thm-restate}

\newcommand{\norm}[1]{\left\lVert#1\right\rVert}
\usepackage{setspace}
\usepackage{relsize}
\makeindex 

\newcommand*{\medcup}{\mathbin{\scalebox{1.5}{\ensuremath{\cup}}}}%

\newtheorem{thm}{Theorem}[section]
\newtheorem{cor}[thm]{Corollary}
\newtheorem{lemma}[thm]{Lemma}
\newtheorem{prop}[thm]{Proposition}

\theoremstyle{definition}
\newtheorem{defn}[thm]{Definition}

\newtheorem*{FROS}{Frostman's Lemma}
\newtheorem*{FALthm}{Falconer's Theorem}
\newtheorem{FSthm}[thm]{Fraser-Shmerkin's Theorem}

\theoremstyle{definition}

\newtheorem*{structure}{Structure and ideas of the article}
\newtheorem{rem}[thm]{Remark}

\usepackage{geometry}
\usepackage{graphicx}
\usepackage{amssymb}
\usepackage{epstopdf}
\usepackage{enumerate}
\numberwithin{equation}{section}
\begin{document}

\title{Dimensions of an overlapping generalization of Bara\'{n}ski carpets}

\author{Leticia Pardo Sim\'{o}n}

\date{\today}

\maketitle

\begin{abstract}
We determine the Hausdorff, packing and box-counting dimension of a family of self-affine sets generalizing Bara\'{n}ski carpets. More specifically, we fix a Bara\'{n}ski system and allow both vertical and horizontal random translations, while preserving the rows and columns structure. The alignment kept in the construction allows us to give expressions for these fractal dimensions outside of a small set of exceptional translations. Such formulas will coincide with those for the non-overlapping case, and thus provide examples where the box-counting and Hausdorff dimension do not necessarily agree. These results rely on M. Hochman's recent work on the dimensions of self-similar sets and measures, and can be seen as an extension of J. Fraser and P. Shmerkin results for Bedford-McMullen carpets with columns overlapping. 
\end{abstract}
 
\section{Introduction}
Frequently, we find that fractals are comprised of scaled-down copies of themselves, which permits them to be represented as attractors of iterated function systems. Recall that an \textit{iterated function system (IFS)} is a finite family $\lbrace S_{i} \rbrace_{i \in \mathcal{I}}$ of contractions defined on a closed subset $D \subseteq \mathbb{R}^n$, i.e. functions that satisfy
$\vert S_{i}(x)-S_{i}(y) \vert \leq c_{i} \vert x-y \vert $ for all $x,y \in D$ and some $c_{i}<1 $. Hutchinson \cite{HUCH} proved in 1981 that given an IFS, there exists a unique non-empty compact set $F$, called its \textit{attractor}, that satisfies
\begin{equation}\label{attractor} F = \bigcup_{i \in \mathcal{I}} S_{i} (F). \\
\end{equation}

When aiming to compute fractal dimensions, this representation turns out to be very convenient, and in fact the study of dimensions of attractors of IFSs has been a long standing problem. In particular, if all the contractions that form an IFS are similarities, that is, $\vert S_{i}(x)- S_{i}(y) \vert = c_{i} \vert x-y\vert $ for all $i \in \mathcal{I}$, the corresponding attractor is called a \textit{self-similar set}. More generally, if all maps are affine, i.e. consisting of a linear part and a translation vector, the associated attractors are known as \textit{self-affine sets}. This paper will study certain class of self-affine sets, but will make use of results on self-similar sets.\\

Given an IFS of similarities, we say that the \textit{open set condition} (OSC) holds if there exists a non-empty open set $U$ such that $\bigcup_{i \in \mathcal{I}} S_{i} (U) \subseteq U$ with this union disjoint, and thus guaranteeing that the union in (\ref{attractor}) is ``almost disjoint''. Under this separation condition, already back in 1946 P. Moran \cite{MORAN} presented a formula for computing the ``size'' of self-similar sets. The \textit{similarity dimension} is defined to be the unique solution $s$ to the equation \begin{equation} \label{sim}
\sum_{i \in \mathcal{I}} c_i^{s}=1, 
\end{equation}
and equals both the Hausdorff and box-counting dimension of the attractor of the system.\\

However, when the OSC is not satisfied, finding general expressions for the dimensions of self-similar sets becomes a trickier task. In $\mathbb{R}$, a ‘dimension drop’ can occur if the image of different iterates of some maps of the IFS overlap exactly, and it has been conjectured for a long time that this is the only way the dimension can drop, see for example \cite{PS}. Recently, an important step towards solving this conjecture has been made by Hochman \cite{Ho}, who confirms it in the case where the defining parameters of the IFS are algebraic. We will make use of this result in our proofs. When working in higher dimensions, the
conjecture above is false as stated, and a new version which pays attention to the case when the linear parts of the defining similarities act reducibly on
$\mathbb{R}^{d}$ is formulated in \cite{Ho2}. \\

Self-affine sets follow a more complex behaviour and consequently are not so well understood. To begin with, the Hausdorff dimension need not vary continuously with the parameters even when the OSC is satisfied, see \cite{Falconer, Lal, Przytycki1989}. Thus, the expectations of finding dimension formulas as treatable as (\ref{sim}) are lower. Nevertheless, a first general result for maps whose linear parts are nonsingular and contractive was due to Falconer in 1988. He introduced the so-called \textit{affinity dimension d}, given in terms of the singular values of these linear parts (for its definition see \cite[Section 4 and Theorem 5.3]{Falconer}). The main theorem is as follows: 
\begin{FALthm}\cite[Theorem 5.3]{Falconer}. Suppose that each of the linear maps $\lbrace A_{i} : i \in \mathcal{I}\rbrace$ satisfies $\norm{ A_{i}} < \frac{1}{3}$. Then for almost all $\underline{t}\in \mathbb{R}^{n\vert \mathcal{I} \vert }$ (in the sense of the $n$-dimensional Lebesgue measure) the attractor $F_{\underline{t}}$ of the IFS $\lbrace A_{i} + t_{i}\rbrace_{i \in \mathcal{I}}$ satisfies $\dim_{H} F_{\underline{t}}= \dim_{B} F_{\underline{t}} = \min \lbrace n, d \rbrace$.\\
\end{FALthm}

The condition on the norm of the maps was relaxed to $1/2$ by Solomyak \cite{PSP:37417}, who also noted that $1/2$ is sharp based on an example of Przytycki and Urba\'{n}ski \cite{Przytycki1989}. Note that Falconer's setting does not have any restriction with regard to alignments nor overlaps, but unfortunately, the proof of the theorem does not give any information as to which $\underline{t}$ the formula applies. This originated a line of research aiming to establish sufficient conditions for the validity of the theorem, as well as extending it; see for example\cite{Hueter95, ETS:5571624, Jordan2007, Shmerkin, 0951-7715-12-4-308}. Besides, it is a difficult problem to actually compute $d$  in most cases.\\ 

Thanks to the seminal work on specific cases by Bedford \cite{Bedford} and McMullen \cite{McMullen}, it was already known in 1984 that the equality on the dimensions stated in Falconer's Theorem does not hold for all parameters $\underline{t}$. The dynamical construction of their setting is as follows: they divided the unit square into a uniform grid of $m\times n$ equal rectangles for some fixed $n>m$ integers. This grid can be naturally labelled as $D_0=\{(i,j):1\leq i\leq m\text{ and }1\leq j\leq n\}$. Then they chose a subset $D \subset D_0$ and considered the IFS consisting on the affine transformations which map $[0, 1]^{2}$ onto each rectangle in $D$, preserving orientation; see Figure \ref{fig:setting}. The uniformity of the model allowed them to provide explicit formulae for the Hausdorff, packing and box-counting dimensions of the corresponding attractor $F$, namely
\begin{equation}\label{dims}
\dim_{H} F =\dfrac{\log \sum_{i \in  \overline{D}_{X} }N_{i}^{\frac{\log m}{\log n}}}{\log m} \quad \quad \dim_{B} F=\dim_{P} F=\frac{\log \vert \overline{D}_{X}\vert }{\log m}+\dfrac{\log (\vert D \vert /\vert \overline{D}_{X}\vert )}{\log n}, 
\end{equation}
where $\overline{D}_{X} = \lbrace i \in \lbrace 1, \ldots ,m\rbrace : (i,j) \in D \text{ for some } j\rbrace$ denotes the projection of $D$ onto the horizontal axis, and $N_{i}$ represents the number of rectangles in the $i$th column that belong to $D$.  We shall refer to this family of attractors as \textbf{Bedford-McMullen carpets}.\\

Note that for most choices of the set $D$, the Hausdorff and box-counting dimension will be different from each other, with the equality holding when all non-empty columns have the same number of elements. Similar phenomena occur in more general \textit{carpets}, that is, attractors of systems defined by a pattern $D$ of (not necessarily equal) rectangles in the unit square. Due to the importance of this condition in this paper, we give a precise definition. Consider the subsets 
\begin{equation} \label{ij}
I_{i} = \big \{ (k,l) \in D : k=i \big \} \qquad \qquad J_{j} =  \big \{ (k,l) \in D : l=j \big \}. 
\end{equation}
\vspace{-15pt}
\begin{defn} A carpet (or its defining IFS) is said to have \textbf{uniform vertical fibres} if $\vert I_{i} \vert = \vert I_{i'} \vert $ for all  $i, i' \in \overline{D}_{X}$, provided that $I_{i}, I_{i'}\neq \emptyset$. Analogously, a carpet has \textbf{uniform horizontal fibres} if whenever $J_{j},J_{j'}\neq \emptyset$, it holds $\vert J_{j} \vert = \vert J_{j'} \vert $ for all  $j, j' \in \overline{D}_{Y}$ . If the system has both uniform vertical and horizontal fibres, we say that it has \textbf{uniform fibres}.
\end{defn}
\begin{rem}\label{bedrem} We would like to emphasize that usually only uniform vertical fibres are required for some properties to hold, as for example the Hausdorff, packing and box-counting dimension of a Bedford-McMullen carpet are equal if and only if it has uniform vertical fibres. However, the proofs of our results will make use of  Bedford-McMullen-type carpets (see Definition \ref{defT}) with necessarily both uniform horizontal and vertical fibres.
\end{rem}
\begin{figure}[h]
	\centering
		\includegraphics[scale=.38]{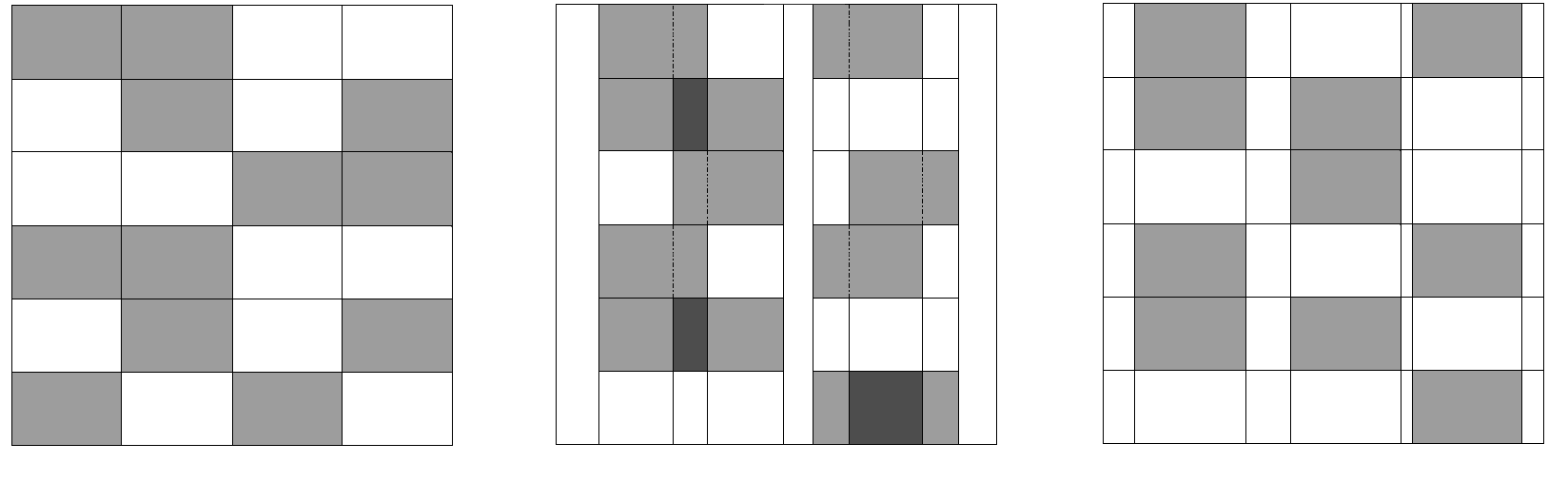}
		\vspace{-30pt}
		  \label{fig:setting}
		  \caption{From left to right: Bedford-McMullen, Fraser-Shmerkin and Bedford-McMullen-type carpets.}
\end{figure}

Following Bedford and McMullen's work, other specific settings with increasing levels of generality were studied: see Gatzouras and Lalley \cite{Lal}, Bara\'{n}ski \cite{Baranski}  or Feng and Wang \cite{Feng} for carpets defined by a pattern of rectangles with non-overlapping interior. Unlike in Bedford-McMullen's setting, in these cases there are no explicit formulae for the Hausdorff dimension, but instead they are given via a variational principle and may be difficult to compute or even estimate. Shmerkin \cite{Shmerkin} considered carpets where overlapping is permitted, obtaining expressions for the dimensions of self-affine sets in certain parametrized families.``Box-like'' sets were Fraser's setting in \cite{Boxlike}, where he relaxed the condition of the maps being orientation-preserving and allowed them to have non-trivial rotational and reflectional components. It is also worth noting the work of D.J. Feng and H. Hu \cite{Feng2} on ergodic properties of IFSs, that in particular relate the Hausdorff dimension of the attractors of certain affine IFSs to projections of ergodic measures. Their results combined with Hochman's work can be used to show that the set of box-like sets where the dimension drops below Falconer's dimension is small. \\

Recently, Fraser and Shmerkin \cite{FS} combined both the general and specific approach on a generalization of Bedford-McMullen carpets, see Figure \ref{fig:setting} for an example. Once the defining pattern of such carpet is fixed, they randomise the vertical translates whilst preserving the column structure intact. As some alignment is kept in their construction, the same formulae (\ref{dims}) as those obtained for Bedford-McMullen carpets hold for all translation parameters except for a small exceptional set. Therefore, they provide a family that contains many overlapping self-affine sets whose box-counting and Hausdorff dimension are typically different from each other and thus from the affinity dimension.\\

In this paper we extend their results in two directions: on one hand we generalize the systems considered by studying self-affine sets generalizing Bara\'{n}ski carpets, and on the other hand we allow this time simultaneous vertical and horizontal translations, while preserving the rows and columns structure. We will be able to guarantee that for a big set of translation parameters, the potentially generated overlaps do not cause the dimensions of the new attractors to fall below of the dimensions (generally different from each other) of the attractor of the original system.  As a corollary, we obtain the same corresponding result for Bedford-McMullen carpets, this time with overlapping in two directions.\\

We believe that the significance of this work comes not only from providing a large family of overlapping self-affine sets which fail to satisfy the dimension equalities in Falconer's theorem, but also from showing that the recent results of Hochman for self-similar sets in $\mathbb{R}$ have consequences for self-affine sets in $\mathbb{R}^{2}$ that overlap in more than one direction.

\subsection{Our setting}
Fix positive integers $m$, $n$ and consider a partition of the unit square into $m \times n$ rectangles: for each $1\leq i\leq m$ and $1\leq j\leq n$ we fix values $0<a_i, b_j<1$ such that $\sum_{i=1}^m a_i=\sum_{j=1}^n b_j=1$, and divide the square $[0,1]^{2}$ into $m$ vertical strips of widths $a_1, \ldots, a_m$  and $n$ horizontal strips of heights $b_1, \ldots , b_n$. Let $D_0=\{(i,j):1\leq i\leq m\text{ and }1\leq j\leq n\}$. 

\begin{defn} Given a subset $D \subsetneq D_0$, we will call  the IFS $\lbrace S_{(i,j)} \rbrace_{(i,j) \in D}$ a \textbf{Bara\'{n}ski system} when for each $(i,j)\in D$,
 \vspace{-5pt} 
\begin{equation*}
S_{(i,j)}\begin{pmatrix}
 x \\
 y \\
\end{pmatrix}=\begin{pmatrix}
				a_i & 0 \\
				0 & b_j \\
				\end{pmatrix} \begin{pmatrix}
 									x \\
 									y \\
								\end{pmatrix}	+ \begin{pmatrix}
 									 \sum_{l=1}^{i-1}a_l \\
 									 \sum_{l=1}^{j-1}b_l \\
								\end{pmatrix}
\end{equation*}
is an affine transformation that maps the unit square onto a translated rectangle of width $a_i$ and height $b_j$. The corresponding attractor $F$ will be a \textbf{Bara\'{n}ski carpet}.
\end{defn}
\medskip
In \cite{Baranski}, Bara\'{n}ski computed the Hausdorff and box-counting dimension of these attractors. For our setting, given a Bara\'{n}ski system, we randomise both the horizontal and vertical translates in the described system, whilst preserving the rows and columns structure. That is, if two rectangles of $D$ are initially in the same row (resp. column), then they are translated horizontally (resp. vertically) by the same amount. See Figure \ref{fig:bar}.
\begin{figure}[h]
	\centering
		\includegraphics[scale=.4]{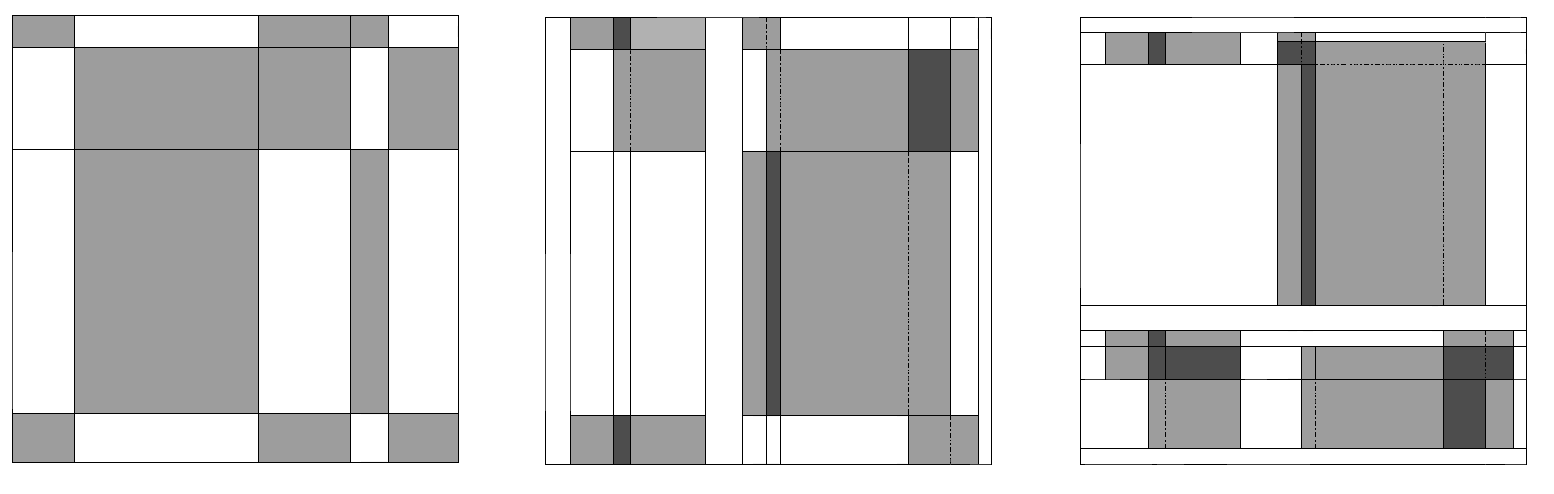}
		\vspace{-15pt}
		  \caption{From left to right: a Bara\'{n}ski carpet and two examples of our setting.}
	 \label{fig:bar}
\end{figure}

More formally, let $\overline{D}_{X} = \lbrace i \in \lbrace 1, \ldots ,m\rbrace : (i,j) \in D \text{ for some } j\rbrace$  and $\overline{D}_{Y} = \lbrace j \in \lbrace 1, \ldots ,n\rbrace : (i,j) \in D \text{ for some } i\rbrace$ denote the projections of $D$ onto the $X$ and $Y$ axes. To each $(i,j) \in (\overline{D}_{X},\overline{D}_{Y})$ we associate a ``random translation'' $(t_{i},\tau_{j}) \in [0, 1 - a] \times [0, 1 - b] $, where $a=\max_{i} a_{i}$ and $b=\max_{j} b_{j}$. We denote the set of all possible translation parameters by
\begin{equation*}
A := \left[0, 1-a \right]^{\vert \overline{D}_{X}\vert}  \times \left[0, 1- b \right]^{\vert \overline{D}_{Y}\vert}, 
\end{equation*} 
with $\vert \overline{D}_{X}\vert$, $\vert \overline{D}_{Y}\vert$ being the cardinals of the projections on the horizontal/vertical axis, i.e the corresponding number of non-empty columns/rows. For each given vector of translates $\underline{t}=(\underline{t}_X, \underline{\tau}_Y)\in A$  we define a new IFS consisting of the maps
\begin{equation*}
S_{\underline{t},(i,j)}\begin{pmatrix}
 x \\
 y \\
\end{pmatrix}=\begin{pmatrix}
				a_i & 0 \\
				0 & b_j \\
				\end{pmatrix} \begin{pmatrix}
 									x \\
 									y \\
								\end{pmatrix}	+ \begin{pmatrix}
 									 t_{i} \\
 									 \tau_{j} \\
								\end{pmatrix},
\end{equation*}
and denote by $F_{\underline{t}}$  its  associated attractor. The reason why we let the parameters $(t_i,\tau_{j}) \in  [0, 1-a]\times [0, 1-b]$ instead of $[0,1]^{2}$ is in order to ensure that $F_{\underline{t}}$ is a subset of the unit square, but this is not an essential requirement.\\

Observe that in the special case when $n \geq m$ and $a_i=1/m$ and $b_j=1/n$ for all $1\leq i\leq m$ and $1\leq j\leq n$, the Bara\'{n}ski system is in fact a Bedford-McMullen one. More generally and by analogy:

\begin{defn}\label{defT} Let $\mathcal{I}$ be a Bara\'{n}ski system such that there are real numbers $\tilde{n}> \tilde{m} >1$ for which $a_i=1/\tilde{m}$ and $b_j=1/\tilde{n}$ for all $1\leq i\leq m$ and $1\leq j\leq n$ whenever $I_i, J_j \neq \emptyset$. Then we call $\mathcal{I}$ a \textbf{Bedford-McMullen-type system} of parameters $(\tilde{m}, \tilde{n})$. See Figure \ref{fig:setting}.\\
\end{defn}

\noindent For convenience in forthcoming arguments and statements of results, we have assumed that $\tilde{n}>\tilde{m}$, but the symmetric case presents analogous conclusions. For Bedford-McMullen-type carpets with possible overlapping columns we have the following result regarding their dimensions: 

\begin{FSthm}\cite[Theorem 7.1]{FS}\label{FSthm} Let $F_{\underline{t}}$ be a Bedford-McMullen-type carpet with $\tilde{n}>\tilde{m}\geq 2$ and $\underline{t} \in A$ any vector such that the IFS $\lbrace x/\tilde{m}+t_i\rbrace_{ i \in \overline{D}_{X}}$ does not have super exponential concentration of cylinders (see Definition \ref{concent}) and $\lbrace x/\tilde{n}+\tau_j\rbrace_{ j \in \overline{D}_{Y}}$ satisfies the OSC. Then it holds 
\begin{align*}
\dim_{H}(F_{\underline{t}})=&\dfrac{\log \sum_{i \in \overline{D}_{X}}\vert I_{i} \vert^{\frac{\log \tilde{m}}{\log \tilde{n}}}}{\log \tilde{m}}, \\
\dim_{B}(F_{\underline{t}})=&\frac{\log \vert \overline{D}_{X}\vert }{\log \tilde{m}}+\dfrac{\log (\vert D\vert/\vert \overline{D}_{X}\vert )}{\log \tilde{n}}.
\end{align*}
\end{FSthm}

\subsection{Statement of results}

Bara\'{n}ski proved for his attractors that their Hausdorff dimension is given by the maximum value that a function $g$ takes over the set $\mathbb{P}^{|D|}$ of probability vectors; see Subsection \ref{subsecupperH} for concrete definitions. We are able to achieve in our case exactly the same result for a big subset of the translation parameters $\underline{t} \in A$. Recall that an IFS $\lbrace S_{1}, \ldots S_{k}\rbrace  $ is said to have an \textbf{exact overlap} if the semigroup generated by the $S_{i}$ is not free, and we will say that $\underline{t} \in A$ is \textit{algebraic} if all of its coordinates $t_i$, $\tau_j$ are algebraic.

\begin{thm} \label{THMhaus} For each Bara\'{n}ski system there exists a set $E \subseteq A $ of Hausdorff and packing dimension $|\overline{D}_{Y}|+ |\overline{D}_{X}| - 1$ (in particular of zero $\vert \overline{D}_{X} \vert +\vert \overline{D}_{Y} \vert $-dimensional Lebesgue measure) such that
\begin{align*}
\dim_{H}(F_{\underline{t}})=\max_{\textbf{p}\in\mathbb{P}^{|D|}}g(\textbf{p})\quad \quad &\text{if } \underline{t} \in A \setminus E\\
\dim_{H}(F_{\underline{t}})\leq \max_{\textbf{p}\in\mathbb{P}^{|D|}}g(\textbf{p}) \quad \quad &\text{if } \underline{t} \in E
\end{align*}
Furthermore, if all the defining parameters $a_i$, $b_j$ and the vector $\underline{t}$ are algebraic and the IFSs $\lbrace a_{i}x + t_{i}\rbrace_{ i \in \overline{D}_{X}}$ and $\lbrace b_{j} y+ \tau_{j}\rbrace_{ j \in \overline{D}_{Y}}$ do not have an exact overlap, then $\underline{t} \notin E$.\\
\end{thm}

For a general fixed Bara\'{n}ski system, we are not able to assert that there is a dimension drop for the attractors associated to the parameters in the exceptional set $E$. Nonetheless, the geometry of the Bedford-McMullen-type systems allows us to guarantee the existence of such a ``sharp'' exceptional set:

\begin{cor}\label{corHauss} For each Bedford-McMullen-type system of parameters $(\tilde{m},\tilde{n})$, with $\tilde{n}>\tilde{m}$, there exists a set $E_0 \subseteq A $ of Hausdorff and packing dimension $|\overline{D}_{Y}|+ |\overline{D}_{X}| - 1$ (in particular of zero $\vert \overline{D}_{X} \vert +\vert \overline{D}_{Y} \vert $-dimensional Lebesgue measure) such that
\begin{align*}
 \dim_{H}(F_{\underline{t}})=\dfrac{\log \sum_{i \in \overline{D}_{X}}\vert I_{i} \vert^{\frac{\log \tilde{m}}{\log \tilde{n}}}}{\log \tilde{m}} \qquad &\text{if } \underline{t} \in A \setminus E_0 \\
 \dim_{H}(F_{\underline{t}})<\dfrac{\log \sum_{i \in \overline{D}_{X}}\vert I_{i} \vert^{\frac{\log \tilde{m}}{\log \tilde{n}}}}{\log \tilde{m}} \qquad & \text{if } \underline{t} \in  E_0
\end{align*}
Furthermore, if $\tilde{n},\tilde{m}$ and the vector $\underline{t}$ are algebraic and the IFSs $\lbrace x/\tilde{m} + t_{i}\rbrace_{ i \in \overline{D}_{X}}$ and $\lbrace y/\tilde{n} + \tau_{j}\rbrace_{ j \in \overline{D}_{Y}}$ do not have an exact overlap, then $\underline{t} \notin E_0$.\\
\end{cor}

\noindent Similarly, Bara\'{n}ski's formulae for the box-counting dimension holds in our case for a large set of translation vectors: 

\begin{thm} \label{THMbox} For each Bara\'{n}ski system there exists a set $E \subseteq A $ of Hausdorff and packing dimension $|\overline{D}_{Y}|+ |\overline{D}_{X}| - 1$ (in particular of zero $\vert \overline{D}_{X} \vert +\vert \overline{D}_{Y} \vert $-dimensional Lebesgue measure) such that
\begin{align*}
\dim_{B}(F_{\underline{t}})=\dim_{P}(F_{\underline{t}})=\max (D_{A}, D_{B})  \qquad & \text{if } \underline{t} \in A \setminus E\\
\dim_{B}(F_{\underline{t}})=\dim_{P}(F_{\underline{t}})\leq \max (D_{A}, D_{B})  \qquad & \text{if } \underline{t} \in E 
\end{align*}
where $D_{A}$, $D_{B}$ are the unique real numbers such that
\begin{equation}\label{DADB}
\sum_{(i,j)\in D} a_{i}^{t_{A}}b_{j}^{D_{A}-t_{A}}=1, \qquad \qquad \sum_{ (i,j)\in D} b_{j}^{t_{B}}a_{i}^{D_{B}-t_{B}}=1, 
\end{equation}
and $t_{A}$, $t_{B}$ are the unique real numbers such that 
\begin{equation}\label{TB}
\sum_{ i \in \overline{D}_{X} } a_{i}^{t_{A}}=1, \qquad \qquad \sum_{j \in \overline{D}_{Y}} b_{j}^{t_{B}}=1. 
\end{equation}

\noindent Furthermore, if all the defining parameters $a_i$, $b_j$ and the vector $\underline{t}$ are algebraic and the IFSs $\lbrace a_{i}x + t_{i}\rbrace_{ i \in \overline{D}_{X}}$ and $\lbrace b_{j} y+ \tau_{j}\rbrace_{ j \in \overline{D}_{Y}}$ do not have an exact overlap, then $\underline{t} \notin E$.\\
\end{thm}

\begin{cor}\label{corBox} For each Bedford-McMullen-type system of parameters $(\tilde{m},\tilde{n})$, with $\tilde{n}>\tilde{m}$, there exists a set $E_1 \subseteq A $ of Hausdorff and packing dimension $|\overline{D}_{Y}|+ |\overline{D}_{X}| - 1$ (in particular of zero $\vert \overline{D}_{X} \vert +\vert \overline{D}_{Y} \vert $-dimensional Lebesgue measure) such that
\begin{align*}
 \dim_{B}(F_{\underline{t}})=\dim_{P}(F_{\underline{t}})=\frac{\log \vert \overline{D}_{X}\vert }{\log \tilde{m}}+\dfrac{\log (\vert D\vert/\vert \overline{D}_{X}\vert )}{\log \tilde{n}} \qquad &\text{if } \underline{t} \in A \setminus E_1\\
 \dim_{B}(F_{\underline{t}})=\dim_{P}(F_{\underline{t}})<\frac{\log \vert \overline{D}_{X}\vert }{\log \tilde{m}}+\dfrac{\log (\vert D\vert/\vert \overline{D}_{X}\vert )}{\log \tilde{n}} \qquad &\text{if } \underline{t} \in E_1
\end{align*}
Furthermore, if $\tilde{n},\tilde{m}$ and the vector $\underline{t}$  are algebraic and the IFSs $\lbrace x/\tilde{m} + t_{i}\rbrace_{ i \in \overline{D}_{X}}$ and $\lbrace y/\tilde{n} + \tau_{j}\rbrace_{ j \in \overline{D}_{Y}}$  do not have an exact overlap, then $\underline{t} \notin E_1$.\\
\end{cor}

\begin{rem} The exceptional set $E$ in Theorems \ref{THMhaus} and \ref{THMbox} depends on the defining parameters $a_i, b_j$ of the fixed Bara\'{n}ski system. Nonetheless, $E$ happens to be the same set in both theorems when working with the same original system. See equation (\ref{setE}) for its definition. However, the sets $E_0$ and $E_1$ in the corollaries are not necessarily equal, and in fact they will not be in most cases.\\
\end{rem}

\begin{structure} We start by establishing some symbolic notation in Section \ref{measures}, in addition to describing those results due to Hochman that will play a key role in our proofs. Section \ref{sectionHaus} deals with our results concerning the Hausdorff dimension, i.e, Theorem \ref{THMhaus} and Corollary \ref{corHauss}. We will firstly discuss how Bara\'{n}ski's argument for getting an upper bound adapts to our setting, and then we will estimate the lower bound through controlled approximations: firstly to a Bedford-McMullen-type subsystem, and then using Hochman's results to a new subsystem without overlapping rows. The new subsystems will have ``enough maps'' as to give us the desired bound by applying Fraser-Shmerkin's Theorem. Finally, Section \ref{sectionBoxC} addresses the calculation of the box-counting dimension, Theorem \ref{THMbox} and Corollary \ref{corBox}, for which an upper bound is provided by Fraser's work \cite[Theorem 2.4]{Boxlike} on box-like sets. A lower bound will be estimated following a similar reasoning to that for the Hausdorff dimension. However, this time we will have to perform approximations until we get a system without any overlaps, since the dimension will be computed  by estimating the number of squares of a same size required to cover the image of the original carpet under the final approximating subsystem.
\end{structure}

\subsection*{Acknowledgements}
I am especially grateful to my supervisor Thomas Jordan for introducing me to this problem and for his continuous help and guidance, as well to the referee for many helpful and detailed comments. I would also like to thank Alexandre De Zotti and Lasse Rempe-Gillen for their valuable suggestions.  

\section{Symbolic notation and self-similar measures} \label{measures} A direct correspondence between our attractors and certain symbolic spaces will allow us to work with the usually simpler geometry of the latter, as well as transferring properties between spaces. We start by setting some notation. For $\mathbold{\lambda_{k}}=(\lambda_{1},\lambda_{2},\ldots ,\lambda_{k})=((i_{1},j_{1}),\ldots ,(i_{k},j_{k})) \in D^{k}$ and fixed $\underline{t} \in A$, we denote the composition of the associated maps by
\[S_{\underline{t},\mathbold{\lambda_{k}}}=S_{\underline{t},(i_{1},j_{1})}\circ \cdots \circ S_{\underline{t},(i_{k},j_{k})}.\]

\noindent The image of the unit square under these maps will be represented by $$\Delta_{\underline{t},\mathbold{\lambda_{k}}}=S_{\underline{t},\mathbold{\lambda_{k}}}([0,1]^{2}),$$whose respective width and height are 
\[ A_{\mathbold{\mathbold{\lambda_{k}}}}:=a_{i_{1}}\cdots a_{i_{k}}  \qquad \qquad  B_{\mathbold{\lambda_{k}}}:=b_{j_{1}}\cdots b_{j_{k}}.\] 

\noindent As auxiliary variables we define
\[T_{\mathbold{\lambda_{k}}}:=\min(A_{\mathbold{\lambda_{k}}},B_{\mathbold{\lambda_{k}}}) \qquad L_{\mathbold{\lambda_{k}}}:=\max(A_{\mathbold{\lambda_{k}}},B_{\mathbold{\lambda_{k}}}).\]
By convention, $\mathbold{\lambda_{0}}=\emptyset$ and $A_\emptyset=B_\emptyset=1$.

\begin{defn}
We call $\mathbold{\lambda_{k}}$ an \textbf{$\mathbold{A}$-sequence} (resp. $B$-sequence) 
if $ L_{\mathbold{\lambda_{k}}}= A_{\mathbold{\lambda_{k}}}$ (resp. $L_{\mathbold{\lambda_{k}}}=B_{\mathbold{\lambda_{k}}}$).
Let $\mathbold{\lambda_{k}}$ and $\mathbold{\lambda'_{k}}$ be two $\mathrm{A}$-sequences (resp. two $B$-sequences). We say that $\mathbold{\lambda_{k}}$ and $\mathbold{\lambda'_{k}}$ are of the \textbf{same type} if for every $l=1, \ldots $, $k$ we have $i_l=i'_l$ (respectively $j_l=j'_l$). We write $\mathbold{\lambda_{k}} \sim \mathbold{\lambda'_{k}}$ in this situation. Otherwise, we say that $\mathbold{\lambda_{k}}$ and $\mathbold{\lambda'_{k}}$ are of different types.\\
\end{defn}

\noindent Note that two sequences are of the same type if and only if $\Delta_{\underline{t},\mathbold{\lambda_{k}}}$ and $\Delta_{\underline{t},\mathbold{\mathbold{\lambda'_{k}}}}$ are in the same column (resp. row).\\

Given an IFS, for each point $(x,y)$ of its attractor and for any compact set $E$ such that $S_{(i,j)}(E)\subseteq E$, there exists at least one sequence $\mathbold{\lambda}=\lim_{k\rightarrow \infty}\mathbold{\lambda_{k}}=((i_{1}(x),j_{1}(y)),\ldots,(i_{k}(x),j_k(y)),\ldots)$ such that $(x,y) \in S_{\mathbold{\lambda_{k}}}(E)$ for all $k$. (See \cite[Chapter 9]{FALbook}). In particular, $\lbrace (x,y)\rbrace=\bigcap^{\infty}_{k=1}S_{\underline{t}, \mathbold{\lambda_{k}}}(E)$. Since our functions are uniformly contracting, we can then write
\begin{equation} \label{limitx}
(x,y)=\sum_{k=1}^{\infty} \left(A_{\mathbold{\lambda_{k-1}}} t_{i_{k}(x)} ,B_{\mathbold{\lambda_{k-1}}} \tau_{j_{k}(y)}\right).  
\end{equation}
Thus, if we denote by $D^{\mathbb{N}}$ the set of all sequences of elements of $D$, that is, $D^{\mathbb{N}} = \lbrace (\lambda_{l})_{l=1}^{\infty} : \lambda_{i} \in D \rbrace$,
we get a surjective function that codes our fractal: 
\begin{equation} \label{Pi}
{\setstretch{2}
\begin{matrix}				
\mathlarger{\mathlarger{\Pi_{\underline{t}}}}\: : \:\mathlarger{\mathlarger{D^{\mathbb{N}}}} &\qquad \mathlarger{\mathlarger{\longrightarrow}} & \mathlarger{\mathlarger{F_{\underline{t}}}} \\
				\qquad \mathbold{\lambda} & \quad \rightsquigarrow & \sum_{k=1}^{\infty} \left( A_{\mathbold{\lambda_{k-1}}} t_{i_{k}(x)} ,B_{\mathbold{\lambda_{k-1}}} \tau_{j_{k}(y)}\right)  \\
\end{matrix}}
\end{equation}

\noindent The map $\Pi_{\underline{t}}$ will allow us to induce a measure on $F_{\underline{t}}$ from a suitable measure on $ D^{\mathbb{N}}$.\\

\begin{defn} 
A \textbf{cylinder of level $ \mathbold{k}$} in $D^{\mathbb{N}}$ is a set of the form 
\[[\lambda_{1},\ldots,\lambda_{k}]=\lbrace \mathbold{\omega}=(\omega_{l})_{l=1}^{\infty}  \in D^{\mathbb{N}} \text{  such that  } \omega_{l}=\lambda_{l} \text{  for all  } 1\leq l \leq k  \rbrace.\]
We will also use the notation $\mathcal{C}_{\mathbold{\lambda_k}}$ to refer to the same set. Besides, for two sequences $\mathbold{\lambda}$  and $\mathbold{\omega} \in D^{\mathbb{N}}$, $\vert \mathbold{\lambda} \wedge \mathbold{\omega} \vert =\min \lbrace k : \lambda_k \neq \omega_{k} \rbrace$ is the index of the first coordinate in which the sequences differ.\\
\end{defn}

Note that a cylinder of level $k$ is the set of all sequences that agree on some  specific first $k$ terms. Cylinders can generate a $\sigma$-algebra $\mathcal{A}$ that makes the pair $(D^{\mathbb{N}}, \mathcal{A})$ a measurable space, over which we can define a class of measures:  

\begin{defn}\label{Bermeasures}
Given a probability vector $\mathbold{p}=(p_{1},\ldots, p_{\vert D \vert})$, the \textbf{Bernoulli measure} with weights $\mathbold{p}$ is the measure $\nu_{\mathbold{p}}$ which assigns to each cylinder $[\lambda_{1},\ldots,\lambda_{k}]$ the value
\[\nu_{\mathbold{p}}([\lambda_{1},\ldots,\lambda_{k}])=p_{\lambda_{1}}p_{\lambda_{2}}\cdots p_{\lambda_{k}},\]
where a bijection from the the naturals $1,2, \ldots, \vert D \vert$ to the elements $\lambda_{l} \in D$ has been defined.\\
\end{defn}

For any Bernoulli measure $\nu_{\mathbold{p}}$ defined on $D^{\mathbb{N}}$, we can define a measure $\mu_{\mathbold{p}}$ on our fractal $F_{\underline{t}}$ as the pushforward measure of $\nu_{\mathbold{p}}$ by the map $\Pi_{\underline{t}}$ given in (\ref{Pi}): 
$$ \mu_{\mathbold{p}}:=\nu_{\mathbold{p}} \circ \Pi_{\underline{t}}^{-1}.$$

Observe that the function $\Pi_{\underline{t}}$ can be rewritten as $\Pi_{\underline{t}}(\mathbold{\lambda})= \bigcap_{k=1}^{\infty}\Delta_{\underline{t},\mathbold{\lambda_{k}}}$. Then, since $\Pi_{\underline{t}}$ is not necessarily injective because the sets $\Delta_{\underline{t},\mathbold{\lambda_{k}}}$ can intersect, we can only guarantee that $[\lambda_{1},\ldots,\lambda_{k}]\subset \Pi_{\underline{t}}^{-1}(\Delta_{\underline{t},\mathbold{\lambda_{k}}})$. Therefore,
\begin{equation}\label{measureF}
\mu_{\mathbold{p}}(\Delta_{\underline{t},\mathbold{\lambda_{k}}}) \geq \nu_{\mathbold{p}}([\lambda_{1},\ldots,\lambda_{k}])= p_{\lambda_{1}}p_{\lambda_{2}}\cdots p_{\lambda_{k}}.
\end{equation}\\
\noindent We now present results due to Hochman \cite{Ho} on the dimensions of self-similar sets supported on $\mathbb{R}$. Let $$\mathcal{I} =\lbrace S_{i}(x) = c_i x+t_{i}\rbrace_{i \in B},$$ with $B$ a finite index set, $c_i\in (0,1)$ for all $i \in B$, and the maps $S_{i}$ acting on $\mathbb{R}$. Hochman's theorem will exclude some ``problematic'' IFS's, namely:
\begin{defn}\label{concent}
 We say that  the IFS $\mathcal{I}$ has \textbf{super-exponential concentration of cylinders (SECC)} if $-\log\gamma_{k}/k \rightarrow \infty$  (with the convention $\log 0 = -\infty$), where
\[\gamma_{k}=\min_{\mathbold{\lambda_k}\neq \mathbold{\lambda'_k}} \vert S_{\mathbold{\lambda_k}}(0)-S_{\mathbold{\lambda'_k}}(0) \vert\]
and $S_{\mathbold{\lambda_k}}(x) = S_{i_{1}} \circ \cdots \circ S_{i_{k}}(x)\text{ if } \mathbold{\lambda_k}= (i_{1},\ldots,i_{k}) \in B^{k}$.\\
\end{defn}

That is, $\gamma_{k}$ records the minimum distance between different $k$-cylinders, and Definition \ref{concent} demands the distance to decrease faster than any power as a function of $k$. So far, super-exponential concentration of cylinders are only known to happen when there are exact overlaps, i.e when the the semigroup generated by the defining maps of the IFS $\mathcal{I}$ is not free. \\

Observe that in terms of the generation of the attractor, if such semigroup is not free we will have two different codings for some rectangle of generation $k$, that is, 
$\Delta_{\underline{t},\mathbold{\lambda_{k}}}=\Delta_{\underline{t},\mathbold{\lambda'_{k}}}$ for some $\mathbold{\lambda},\mathbold{\lambda'}$ and $k$. We also note that an exact overlap means that $\gamma_{k}= 0$ for some $k$.\\ 

\begin{rem} \label{tilde} Let $\mathcal{\tilde{I}}$ be an IFS obtained by first iterating a fixed number of times all the maps of $\mathcal{I}$, where $\mathcal{I}$ is an IFS that does not have super-exponential concentration of cylinders, and then removing some of the maps. Then $\mathcal{\tilde{I}}$ does not have super-exponential concentration of cylinders.\\
\end{rem}

\begin{thm}\label{H1}\cite[Corollary 1.2]{Ho}. Suppose the IFS $\mathcal{I} =\lbrace S_{i}(x) = c_i x+t_{i}\rbrace_{i \in B}$ does not have super-exponential concentration of cylinders. Then its attractor $F$ satisfies
\[\dim_{H}F=min\left(s, 1\right),\]
where $s$ is the similarity dimension defined in equation (\ref{sim}).\\
\end{thm}

The following properties also follow from Hochman's work for self-similar sets supported on $\mathbb{R}^{d}$, but we present them in the one-dimensional case, that will suffice in our setting. They tell us that super-exponential concentration of cylinders is a special circumstance; in fact, in some cases its presence is as uncommon as finding exact overlaps.

\begin{prop}  \label{H2} Let $B$ be a finite index set.
\begin{enumerate}
\item The family of $(t_{i})_{i \in B}$ such that $\mathcal{I} =\lbrace S_{i}(x) = c_i x+t_{i}\rbrace_{i \in B}$ has super-exponential concentration of cylinders has Hausdorff and packing dimension $\vert B\vert - 1$.
\item If all the parameters $ \lbrace c_i, t_{i} \rbrace_{i \in B}$ are algebraic, then $\mathcal{I}$ has super-exponential concentration of cylinders if and only if there is an exact overlap, that is, if and only if $\gamma_{k}= 0$ for some $k$.\\
\end{enumerate}
\end{prop}

\begin{proof}
 Let $E$ be the set of parameters  in $\mathbb{R}^{\vert B \vert}$ for with the corresponding IFSs have SECC. Then, for any $i\neq j \in B$, the set $E$ contains the hyperplane $\lbrace \underline{t} :t_i= t_j\rbrace$, and therefore $ \dim_H (E) \geq \vert B \vert -1$. The upper bound follows from \cite[Theorem 1.10]{Ho2}. The second property is \cite[Theorem 1.5]{Ho}.
\end{proof}

We include now a lemma which allows approximations of the dimension
of a self-similar \textit{homogeneous} (i.e. all contraction ratios are the same) system with overlaps by subsystems without overlaps. We say that an IFS $\lbrace S_{i}\rbrace_{i \in B}$ with attractor $F$ satisfies the \textbf{strong separation condition (SSC)} if $S_{i}(F)\cap S_{i'}(F) = \emptyset $ for all distinct $i, i' \in B$.\\

\begin{lemma} \cite[Lemma 6.3]{FS}\label{vitali}
Let $\lbrace S_{i}\rbrace_{i \in B}$ be an IFS of similarities on $[0,1]$, each with the same contraction ratio $a \in (0, 1)$, and with self-similar attractor $F$ having Hausdorff and box-counting dimension $\alpha$, and let $\epsilon > 0$. Then there exists $\ell_{0} \in \mathbb{N}$ such that for all $\ell \geq \ell_{0}$ there exists a subsystem corresponding to a subset $B_{\ell} \subseteq B^{\ell}$ which satisfies the SSC and 
\begin{equation*} \label{eqvitali}
\vert B_{\ell} \vert \geq 3^{-\alpha}a^{-\ell(\alpha-\epsilon)}.
\end{equation*}
\end{lemma}

\noindent Note that the $\ell$ appearing in $B_{\ell}$ indicates dependence on $\ell$, while the $\ell$ appearing in $B^{\ell}$ denotes, as usual, the Cartesian product of $\ell$ copies of $B$.\\ 

\section{Calculation of the Hausdorff dimension}\label{sectionHaus}

We start by introducing what will be the target dimension. Let $d=|D|$ be the cardinal of the set $D$, and consider the spaces of probability vectors 
\[\mathbb{P}^{\vert D \vert}= \bigg \{ (p_l)_{l=1}^{d} \in \mathbb{R}^{d} :  p_1,\ldots, p_d \geq 0, \sum_{l=1}^{d} p_l =1 \bigg\} \quad \text{and} \quad \mathbb{P}_+ ^{\vert D \vert}= \bigg \{ (p_l)_{l=1}^{d} \in \mathbb{P} ^{\vert D \vert} :  p_1,\ldots, p_d >0 \bigg\}.  \]

Then each $\textbf{p}\in\mathbb{P}^{|D|}$ induces a measure on the carpet by assigning a probability $p_{l}$ to each $S_{\underline{t},(i_l,j_l)}([0,1]^{2})$, for $l=1, \ldots ,d$, where a bijection of the set $D$ and the naturals $1, \ldots, d$ has been defined. For convenience we will show the dependence of the probability on the pairs $(i,j)$ by using the notation $p_{ij}$ for the coordinates of $\textbf{p}$. For $1\leq i\leq m$ and for $1\leq j\leq n$ let
$$R_i(\textbf{p})=\sum_{(i,j) \in I_i} p_{ij} \qquad \qquad  S_j(\textbf{p})=\sum_{(i,j) \in J_j} p_{ij}$$
be the respective total probabilities in a column $i$ or row $j$, and consider the subsets of $\mathbb{P}^{|D|}$:
\begin{equation*}
\mathcal{S_{A}}= \bigg\{ \textbf{q} \in\mathbb{P}^{|D|} : \sum_{i=1}^m R_i(\textbf{q})\log a_i \geq \sum_{j=1}^n S_j(\textbf{q})\log b_j         \bigg \}
\end{equation*}

\begin{equation*}
\mathcal{S_{B}}= \bigg\{ \textbf{q} \in\mathbb{P}^{|D|} : \sum_{i=1}^m R_i(\textbf{q})\log a_i \leq \sum_{j=1}^n S_j(\textbf{q})\log b_j         \bigg\}.
\end{equation*}
For any $\textbf{p}\in\mathbb{P}^{|D|}$, define 
\begin{equation*}
   g(\textbf{p})=\left\lbrace
  \begin{array}{l}
     \dfrac{\sum_{i=1}^m R_i(\textbf{p})\log R_i(\textbf{p})}{\sum_{i=1}^m R_i(\textbf{p})\log a_i} +   \dfrac{\sum_{ i=1}^{m} \sum_{(i,j) \in I_i} p_{ij}\log\left(\frac{p_{ij}}{R_i(\textbf{p})}\right)}{\sum_{j=1}^n S_j(\textbf{p})\log b_j} \quad \text{ if } \textbf{p} \in \mathcal{S_{A}}
      \\[15pt]
     
     \dfrac{\sum_{j=1}^n S_j(\textbf{p})\log S_j(\textbf{p})}{\sum_{j=1}^n S_j(\textbf{p})\log b_j}+ \dfrac{\sum_{ j=1}^{n} \sum_{(i,j) \in J_j} p_{ij}\log\left(\frac{p_{ij}}{S_j(\textbf{p})}\right)}{\sum_{i=1}^m R_i(\textbf{p})\log a_i}  \quad \text{ if } \textbf{p} \in \mathcal{S_{B}}
     \setminus \mathcal{S_{A}}. \\
  \end{array}
  \right.
\end{equation*}\\

Note that the function $g$ is well defined (the denominators are non-zero) and continuous, since both sub-functions are continuous and agree in $\mathbold{p}\in (\mathcal{S_{A}}
     \cap \mathcal{S_{B}})$. For more details see \cite[pages 225-226]{Baranski}.

\subsection{Upper bound} \label{subsecupperH}
Intuitively, overlaps shouldn't increase the Hausdorff dimension, so it is licit to expect that Bara\'{n}ski's argument to obtain an upper bound for his original carpets will apply to our setting. Indeed this is the case, so in this subsection we adapt his proof to our construction, pointing out the necessary changes in his proofs. For comparison and detailed proofs we remit the reader to \cite[Sections 3-5]{Baranski}.\\

\noindent Firstly, recall that in Section \ref{measures} we defined $T_{\mathbold{\lambda_{k}}}$ and $L_{\mathbold{\lambda_{k}}}$ as the respective minimum and maximum side-lengths of the rectangle $\Delta_{\underline{t},\mathbold{\lambda_{k}}}$. \\

\begin{defn} For each $k>1$ and fixed $\mathbold{\lambda_{k}} \in D^{k}$, let $\mathbold{\lambda_{l}}=(\lambda_1,\ldots \lambda_l)$ for all $1\leq l\leq k$. Define
\[ M := M(k)=M_{\mathbold{\lambda_{k}}}=\min \lbrace l \leq k : T_{\mathbold{\lambda_{l}}} \leq L_{\mathbold{\lambda_{k}}} \rbrace, \]
where the notations $M$, $M(k)$ will be used when it is clear to which $\mathbold{\lambda_{k}}$ they are associated with.\\
\end{defn}

\noindent Observe that the set $\Delta_{\underline{t},\mathbold{\lambda_{M(k)}}}$ is the first set of the nested sequence $ \big \{ \Delta_{\underline{t},\mathbold{\lambda_{l}}} \big \}_{l=1}^{k}$ for which its shortest edge becomes equal or smaller than the largest edge of $\Delta_{\underline{t},\mathbold{\lambda_{k}}}$. See Figure \ref{fig:apbaran}.

\begin{defn}\cite[Definition 4.3]{Baranski}\label{defU} For a fixed $\mathbold{\lambda_k} \in D^{k}$, let
\[ U_{\mathbold{\lambda_{k}}}=\bigcup \big \{ \mathcal{C}_{\mathbold{\lambda'_{k}}} : \mathbold{\lambda'_{k}} \sim \mathbold{\lambda_{k} \text{ and }  \vert \mathbold{\lambda'_{k}}} \wedge \mathbold{\lambda_{k}} \vert > M_{\mathbold{\lambda_{k}}}   \big \}. \]
Then its \textbf{approximate square} of generation k is defined as
\[\mathcal{Q}_{\mathbold{\lambda_{k}}} = \bigcup \big \{ \Delta_{\underline{t},\mathbold{\lambda'_{k}}} : \mathcal{C}_{\mathbold{\lambda'_{k}}} \subset U_{\mathbold{\lambda_{k}}}  \big \}.  \]
\end{defn}

That is, $U_{\mathbold{\lambda_{k}}}$ is the union of all cylinders of sequences of the same type as $\mathbold{\lambda_{k}}$ that have in common at  least the first $ M_{\mathbold{\lambda_{k}}}$ terms. In other words, the image of all such cylinders, denoted by $\mathcal{Q}_{\mathbold{\lambda_{k}}}$, is contained in $\Delta_{\underline{t},M_{\mathbold{\lambda_{k}}}}$, and since all sequences are of the same type, their corresponding $\Delta_{\underline{t},\mathbold{\lambda'_{k}}}$ will be aligned and have the same width or height, depending on whether they are $A$ or $B$-sequences. See Figure \ref{fig:apbaran}. 

\begin{figure}[h]
	\centering
		\includegraphics[scale=.5]{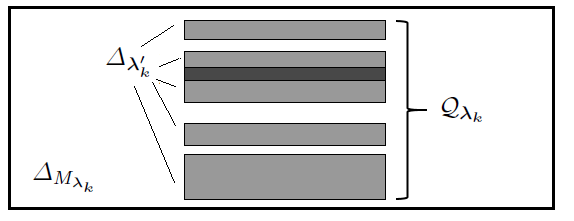}
		 \caption{Definition of an approximate square $\mathcal{Q}_{\mathbold{\lambda_{k}}}$.}
  \label{fig:apbaran}
\end{figure}

\begin{rem}\label{remaprox}
By the definition of $M_{\mathbold{\lambda_{k}}}$, the shortest edge of $\Delta_{\underline{t},M_{\mathbold{\lambda_{k}}}}$ has length $T_{\mathbold{\lambda_{M}}}$, that differs from the longest edge (of length $L_{\mathbold{\lambda_{k}}}$) of the sets $\Delta_{\underline{t},\mathbold{\lambda'_{k}}}$s by a constant, and thus makes it licit to call $\mathcal{Q}_{\mathbold{\lambda_{k}}}$ an approximate square. \\
\end{rem}

\begin{rem}\label{remAdapt} Note that the original Bara\'{n}ski system and our overlapping one have the same number of maps, and thus any argument regarding only symbolic dynamics of the first will also apply in our case. In particular, the number of rectangles in an approximate square is independent of the possible overlapping. The only difference might appear when inducing measures in the attractor of the system. As we pointed out in Section \ref{measures}, the coding function $\Pi_{\underline{t}}$ is not necessarily injective, and thus, by pushing forward we don't get a measure equality but the inequality (\ref{measureF}) instead, although this will suffice for our purposes.\\ 
\end{rem}

We proceed now to sketch the proof to get an upper bound following \cite[Proof of Theorem A]{Baranski}. The goal is to induce a measure in the attractor from an appropriate Bernoulli measure on the symbolic space, for which we can bound the measure of its cylinders. Then, that bound will allow us to use Frostman's Lemma. We have stated it here on a simpler form that will suffice in our case, but a more general version and proof can be found for example in \cite[Theorem 8.8]{Mattila}.
\begin{FROS}Let $\mu$ be a finite Borel measure in $\mathbb{R}^{n}$ and let $A\subseteq \mathbb{R}^{n}$. If for all $x \in A$ 
\begin{equation*}
\liminf_{\delta\rightarrow 0}\frac{\log(\mu_{\textbf{p}}(\mathbb{D}_{\delta}(x))}{ \log \delta}\leq d, \qquad \text{then} \qquad \dim_H(A)\leq d.
\end{equation*}
\end{FROS}
\medskip

In the proof of the following proposition we will work with vectors with positive coordinates. Observe that since $\mathbb{P}_{+}^{\vert D \vert}$ is dense in $\mathbb{P}^{\vert D \vert} \subset \mathbb{R}^{\vert D \vert}$ and the latter is a separable space, there exists a countable set, for example by considering rational coordinates, of vectors 
\begin{equation*}
\lbrace P_L \rbrace_{L=1}^{\infty}  \subset \mathbb{P}_{+}^{\vert D \vert}  \qquad  \text{such that } \qquad \lbrace P_L \rbrace_{L=1}^{\infty} \text{ is dense in } \mathbb{P}^{\vert D \vert}.
\end{equation*}

\begin{prop}\label{upperH} For all $\underline{t} \in A$ it holds 
$dim_{H} F_{\underline{t}}\leq \max\limits_{\textbf{p}\in\mathbb{P}^{|D|}} g(\textbf{p})$.
\end{prop}
\begin{proof}
Fix a small $\epsilon>0$ and let $ (x,y) \in F_{\underline{t}}$. We know by equation (\ref{limitx}) that there exists at least one sequence $\mathbold{\lambda}=\lim_{k\rightarrow \infty}\mathbold{\lambda_{k}}$ such that $(x,y)= \bigcap_{k}
\Delta_{\underline{t},\mathbold{\lambda_{k}}}$. Consider any other sequence $\mathbold{\lambda'_{k}}$ such that $\mathcal{C}_{\mathbold{\lambda'_{k}}}\subset U_{\mathbold{\lambda_{k}}}$. Note that by definition, the longest length of an edge of  $\Delta_{\underline{t},\mathbold{\lambda'_{k}}}$ is $L_{\mathbold{\lambda_{k}}}$,  we have the inclusion $\Delta_{\underline{t},\mathbold{\lambda'_{k}}} \subset\Delta_{\underline{t},M_{\mathbold{\lambda_{k}}}}$, and the 
smallest edge of $\Delta_{\underline{t},M_{\mathbold{\lambda_{k}}}}$ has length $T_{\mathbold{\lambda_{M}}}\leq L_{\mathbold{\lambda_{k}}}$. Thus, we have that all $\Delta_{\underline{t},\mathbold{\lambda'_{k}}}$ are contained in a square of side $L_{\mathbold{\lambda_{k}}}$; see Figure \ref{fig:apbaran}. Therefore, 
$$\mathcal{Q}_{\mathbold{\lambda_{k}}} \subseteq \mathbb{D}_{\sqrt{2}L_{\mathbold{\lambda_{k}}}}((x,y)),$$
and so, using equation (\ref{measureF}), for any $\textbf{p} \in \mathbb{P}^{|D|}$ we have
\begin{equation*}\label{eqmeasures}
\mu_{\textbf{p}}(\mathbb{D}_{\sqrt{2}L_{\mathbold{\lambda_{k}}}}((x,y)) \geq(\nu_{\textbf{p}}\circ \Pi_{\underline{t}}^{-1})( \mathcal{Q}_{\mathbold{\lambda_{k}}}) \geq \nu_{\textbf{p}}( U_{\mathbold{\lambda_{k}}}),
\end{equation*}
and in particular, for all $k$ 
\begin{equation}\label{bdh}
\frac{\log(\mu_{\textbf{p}}(\mathbb{D}_{\sqrt{2}L_{\mathbold{\lambda_{k}}}}((x,y))) }{ \log ( \sqrt{2}L_{\mathbold{\lambda_{k}}})} \leq \frac{\log(\sqrt{2} L_{\mathbold{\lambda_{k}}})-\log\sqrt{2}}{\log (\sqrt{2} L_{\mathbold{\lambda_{k}}})}\cdot\frac{\log \nu_{\textbf{p}}(U_{\mathbold{\lambda_{k}}})}{\log L_{\mathbold{\lambda_{k}}}} .
\end{equation}

In \cite[Proof of Theorem A, pages 233-235]{Baranski}, using purely symbolic arguments, the author finds an upper bound for the limit on $k$ of the latter term in equation (\ref{bdh}). More specifically, he concludes that for every $\mathbold{\lambda} \in D^{\mathbb{N}}$ and each $\epsilon>0$, there exists $\mathbold{p_{_{\mathbold{L(\lambda)}}}} \in \lbrace P_L \rbrace_{L=1}^{\infty} \subset \mathbb{P}_{+}^{\vert D \vert}$ such that
\begin{equation}\label{new2}
 \lim_{k\rightarrow \infty} \frac{\log \nu_{\mathbold{p_{_{\mathbold{L(\lambda)}}}}}(U_{\mathbold{\lambda_{k}}})}{\log L_{\mathbold{\lambda_{k}}}} \leq \max_{\textbf{p}\in\mathbb{P}^{|D|}}g(\textbf{p})+\epsilon.
\end{equation}
By Remark \ref{remAdapt}, such estimate also holds in our case, and this together with equation (\ref{bdh}) gives
\begin{equation*}
\liminf_{\delta\rightarrow 0}\frac{\log(\mu_{\mathbold{p_{_{\mathbold{L(\lambda)}}}}}(\mathbb{D}_{\delta}((x,y)))}{ \log \delta}\leq \max_{\textbf{p}\in\mathbb{P}^{|D|}}g(\textbf{p})+\epsilon.
\end{equation*}

This inequality is valid for all the points in our fractal to which the same vector $P_L$ has been assigned in equation (\ref{new2}). That is, for all points of the set $$F_{\underline{t},L}:=\{(x,y)\in F_{\underline{t}}: \exists\, \mathbold{\lambda} \text{ such that }  \Pi_{\underline{t}}(\mathbold{\lambda})=(x,y) \text{ and } \mathbold{p_{_{\mathbold{L(\lambda)}}}}=\mathbold{p_{_{L}}} \} \text{  for } L=1,2, \ldots. $$
Then $\dim_H (F_{\underline{t},L})\leq \max_{\mathbb{P}^{|D|}} g+\epsilon$ by Frostman's Lemma, so $\dim_H (F_{\underline{t}})=\dim_H(\bigcup_{L=1}^{\infty}F_{\underline{t},L})\leq \max_{\mathbb{P}^{|D|}} g+\epsilon$  by countable stability of the Hausdorff dimension. As $\epsilon$ can be as small as desired, we get $\dim_H(F_{\underline{t}})\leq \max_{\mathbb{P}^{|D|}} g$, which concludes the proof.
\end{proof}

\subsection{Lower bound} \label{sectionLower}
Let $F$ be the attractor of a fixed Bara\'{n}ski system. Our goal in this section is to ensure that $\dim_H (F)=\dim_H(F_{\underline{t}})$ holds for as many attractors $F_{\underline{t}}$, or equivalently as many parameters $\underline{t}\in A$, as possible.
The first step towards this target is to approximate each system in our parametric family of carpets by a sequence $\mathcal{I}_k$ of (possibly overlapping) Bedford-McMullen-type systems with uniform fibres (an idea already used in \cite{proy}). The reason for this is that the property of having uniform fibers will become very useful when getting estimates for the number of maps on further approximations of the system.\\

To each $\mathcal{I}_k$ we associate a number $s_k$ that we prove in Lemma \ref{sktogp} to provide, as $k$ tends to infinity, increasingly good approximations of $\dim_H(F)$. However, we will only be able to guarantee them to be lower bounds of $\dim_H(F_{\underline{t}})$ for certain parameters $\underline{t}$. We define in (\ref{setE}) the set $E$ of ``invalid'' parameters, and for all $\underline{t} \in A \setminus E$, in Lemma \ref{G21} we perform new approximations to subsystems $\mathcal{L}_{\ell}$ by dropping ``not too many'' maps after further iterations of those in $\mathcal{I}_k$. \\

The resulting systems $\mathcal{L}_{\ell}$ with attractors $\Upsilon_{\ell}$ will be free of overlapping rows but they will possibly have columns overlaps. However, since the corresponding projected system on the X-axis does not have SECC, Fraser-Shmerkin's Theorem provides us with a formula for $\dim_H(\Upsilon_\ell)$ in terms of the number of elements on each of its columns. Although we will not know exactly such numbers, the total number of maps of the subsystem is big enough to guarantee that the target dimension is reached for any column distribution of the rectangles.\\

\noindent Let $\textbf{p}$ be the probability vector for which $F$ satisfies $$\dim_H (F)=\max_{\textbf{q}\in\mathbb{P}^{|D|}}g(\textbf{q})= g(\textbf{p}),$$ and let $p_{ij}$ be its coordinates. Without loss of generality we may assume $\textbf{p} \in \mathcal{S}_{A}$ (the other case is symmetric). For $k\in\mathbb{N}$, set $\theta(k) = \sum_{(i,j)\in D} \lceil k p_{i j} \rceil$, and let
\begin{equation*}
 \Gamma_k= \Bigg\{
 \begin{tabular}{c}
 $\mathbold{\lambda_{k}}=(\lambda_{1},\lambda_{2},\ldots ,\lambda_{\theta(k)})\in D^{\theta(k)} : \text{ for all } (i,j) \in D,$ \\
 $ |\lbrace l \in \lbrace 1, \ldots,\theta(k) \rbrace : \lambda_{l}=(i,j) \rbrace|= \lceil k p_{i j}\rceil  $
 \end{tabular}
 \Bigg \},
\end{equation*}
i.e. $\Gamma_k$ is the set of all strings of length $\theta(k)$ over the alphabet $D$ for which the number of occurrences of the pair $(i,j)$ is equal to $\lceil k p_{i j} \rceil$. For each $\underline{t}\in A$ the set $\Gamma_k$ defines an IFS 
\begin{equation}\label{BMcarpet}
\mathcal{I}_k:=\big \{ S_{\underline{t},\mathbold{\lambda_{k}}}\big \}_{\mathbold{\lambda_{k}} \in \Gamma_k}
\end{equation}
with uniform fibres. To see this, let us think of each line (row or column) generated by $\mathcal{I}_k$ as an equivalence class. Then, $S_{\underline{t},\mathbold{\lambda_{k}}}$ and  $S_{\underline{t},\mathbold{\lambda_{k}'}}$ generate elements in the same line when $\lambda_{l}$  and $\lambda'_{l}$ are in the same line for all $1\leq l \leq \theta(k)$. The number of elements on each equivalence class is given by the product of possibilities for each coordinate. But since each element is repeated the same number of times on a sequence $\mathbold{\lambda_{k}}$, the final product is the same for all equivalent classes, that is, all lines have the same number of elements.\\

Besides, note that the number of elements in $\Gamma_{k}$ equals the number of all possible permutations with indistinguishable repetition of the pairs $(i,j) \in D$, each repeated $\lceil k p_{i j} \rceil$ times. Also, in order to get the cardinal of the set  
\begin{equation*}
\overline{\Gamma}^{X}_k:=\big \{(i_1, \ldots, i_{\theta(k)}) : ((i_1,j_1), \ldots, (i_{\theta(k)},j_{\theta(k)})) \in \Gamma_k \text{ for some }j_1, \ldots j_{\theta(k)}\big \}, 
\end{equation*}
that is, the projection of $\mathcal{I}_k$ onto the horizontal axis, we can think of all elements in a column as being identified, and then consider permutations of the columns $i$ with repetition numbers $\sum_{j \in I_i} \lceil k p_{i j} \rceil$. Therefore,

\begin{equation} \label{cardinals}
| \Gamma_k |=\dfrac{\theta(k)!}{\prod_{(i,j)\in D} \lceil k p_{i j}\rceil !} \qquad \qquad \vert \overline{\Gamma
}^{X}_{k} \vert =\dfrac{\theta(k)!}{\prod_{i \in \overline{D}_{X}} \left(\sum_{(i,j)\in I_i} \lceil k p_{i j} \rceil \right)!} .
\end{equation}\\

\noindent Let us denote the attractor associated to $\mathcal{I}_k$ by  $\Lambda_k$. By construction $\Lambda_k \subset F_{\underline{t}}$, and the linear part of each map on $\mathcal{I}_k$ is given by
\[ \text{diag}\left( \prod_{(i,j)\in D} a_{i}^{\lceil k p_{i j} \rceil},\prod_{(i,j)\in D} b_{j}^{ \lceil k p_{i j} \rceil} \right)=:\text{diag}(m_k^{-1},n_k^{-1}).\\ \] 

Since $m_k$ and $n_k$ are not necessarily integers, we have that $\Lambda_k$ is a Bedford-McMullen-type carpet. A simple calculation shows that
\begin{equation}\label{pinB}
\textbf{p} \in \mathcal{S}_{A} \quad \text{implies} \quad n_k \geq m_k.
\end{equation}

\noindent Let us define
\begin{equation}\label{s_k}
s_{k}:=\frac{\log \vert\overline{\Gamma
}^{X}_{k} \vert}{ \log m_k}+ \frac{\log \vert \Gamma_k \vert-\log \vert\overline{\Gamma
}^{X}_{k} \vert}{\log n_k}.
\end{equation}

\begin{lemma}\label{sktogp}
For every $\underline{t} \in A$ there exists a sequence of Bedford-McMullen-type systems $\big \{ \mathcal{I}_{k} \big \}_k$ with uniform fibers and attractors $\Lambda_k \subset F_{\underline{t}}$ for which
\[s_k \longrightarrow \max_{\textbf{q}\in\mathbb{P}^{|D|}}g(\textbf{q} )\] 
as $k$ tends to infinity.
\end{lemma}

\begin{proof}
We will make use a of Stirling's formula for factorials in the following version: for all $b \in \mathbb{N}\setminus \lbrace 1 \rbrace$ we have
\begin{equation} \label{stirh}
b \log b - b \leq \log b! \leq b \log b -b + \log b.\\
\end{equation}

Recall that $\textbf{p}$ is a probability vector and so $\sum_{(i,j)\in D} p_{i j}=1$. This allow us to express $\theta(k) =k \sum_{(i,j)\in D}  p_{i j} + o(k)= k + o(k),$ and for each $i \in \overline{D}_{X}$ we have that $\sum_{(i,j)\in I_i} \lceil k p_{i j} \rceil =k R_{i}(\textbf{p}) + o(k)$. Therefore, the application of Stirling's formula provides
\begin{equation*}
\begin{split}
\lim_{k \rightarrow \infty} \dfrac{\log |\Gamma_{k}|}{k} &\leq \lim_{k \rightarrow \infty} \frac{k\log k - k +\log k-\sum_{(i,j)\in D} \left( k p_{i j} \log k p_{i j}-k p_{i j} \right)}{k} =- \! \! \!\sum_{(i,j)\in D} p_{i j} \log p_{i j}\\
\end{split}
\end{equation*}

\begin{equation*}
\begin{split}
\lim_{k \rightarrow \infty} \dfrac{\log |\Gamma_{k}|}{k} &\geq \lim_{k \rightarrow \infty} \frac{k\log k - k -\sum_{(i,j)\in D} \left( k p_{i j} \log k p_{i j}-k p_{i j} + \log k p_{i j}  \right)}{k} =- \! \! \!\sum_{(i,j)\in D} p_{i j} \log p_{i j}.
\end{split}
\end{equation*}
And thus
\begin{equation}\label{lim11}
\begin{split}
\lim_{k \rightarrow \infty} \dfrac{\log |\Gamma_{k}|}{k}=- \! \! \!\sum_{(i,j)\in D} p_{i j} \log p_{i j}.
\end{split}
\end{equation}
Similarly, 
\begin{equation}\label{lim22}
\begin{split}
\lim_{k\rightarrow \infty} \dfrac{\log \vert \overline{\Gamma
}^{X}_{k}\vert}{k}=- \! \! \!\sum_{i\in \overline{D}_{X}} \! \! \! R_{i}(\textbf{p}) \log R_{i}(\textbf{p}).
\end{split}
\end{equation}
\noindent By definition of $m_k$ and $n_k$,  
\begin{equation*}
\log m_k= - k \! \! \!\sum_{i \in \overline{D}_{X}}  \! \! \!R_i(\textbf{p})\log a_{i}  + o(k) \qquad \log n_k= -k \! \! \!\sum_{j \in \overline{D}_{Y}} \! \! \! S_j(\textbf{p})\log b_j  + o(k).
\end{equation*} 

\noindent Thus, putting all together and recalling equation (\ref{pinB}) and the choice of $\mathbold{p}$ we have
\begin{equation*}
\begin{split}
\lim_{k\rightarrow \infty}s_k &=\frac{\log \vert\overline{\Gamma
}^{X}_{k} \vert}{ \log m_k}+ \frac{\log \vert \Gamma_k \vert-\log \vert\overline{\Gamma
}^{X}_{k} \vert}{ \log n_k}\\
&=\dfrac{\sum_{i \in \overline{D}_{X}} R_{i}(\textbf{p}) \log R_{i}(\textbf{p})}{\sum_{i \in \overline{D}_{X}} R_i(\textbf{p})\log a_{i} }+ \dfrac{\sum_{(i,j)\in D} p_{i j} \log p_{i j}-\sum_{i \in \overline{D}_{X}} R_{i}(\textbf{p}) \log R_{i}(\textbf{p})}{\sum_{j=1}^n S_j(\textbf{p})\log b_j}\\
& = \dfrac{\sum_{i=1}^m R_i(\textbf{p})\log R_i(\textbf{p})}{\sum_{i=1}^m R_i(\textbf{p})\log a_i} +  \dfrac{\sum_{ij} p_{ij}\log\left(\frac{p_{ij}}{R_i(\textbf{p})}\right)}{\sum_{j \in \overline{D}_{Y}} S_j(\textbf{p})\log b_j  }\\
&=g(\textbf{p})= \max_{\textbf{q}\in\mathbb{P}^{|D|}}g(\textbf{q}).
\end{split}
\end{equation*}

\end{proof}

Our strategy relies on projections onto the coordinate axes in order to apply Hochman's results, that will only guarantee the absence of a dimension drop for certain parameters $\underline{t}$. Looking again at our fixed Bara\'{n}ski system, let $E_{X}$ and $E_{Y}$ be the sets of parameters $ \underline{t}_{X} \in [0, 1-a]^{\vert \overline{D}_{X} \vert}$ and  $ \underline{\tau}_{Y} \in [0, 1-b]^{ \vert \overline{D}_{Y} \vert}$ such that the IFSs $\lbrace \overline{S}_{\underline{t},i}\rbrace_{i \in \overline{D}_{X}}$, $\lbrace \overline{S}_{\underline{t},j}\rbrace_{j \in \overline{D}_{Y}}$ have super-exponential concentration of cylinders. Then Theorem \ref{H1} states that any possible overlapping doesn't cause the dimensions of the attractors of the projected systems in the coordinate axes to drop, provided that $ \underline{t}_{X} \in [0, 1-a]^{\vert \overline{D}_{X} \vert} \setminus E_{X}$ and $ \underline{\tau}_{Y} \in [0, 1-b]^{\vert \overline{D}_{Y} \vert} \setminus E_{Y}$. Therefore, the following set stands as the logical candidate for the set of ``invalid'' parameters:
\begin{equation}\label{setE}
E:= \left( E_{X}\times \left[0, 1-b\right]^{\vert \overline{D}_{Y} \vert} \; \medcup \; \left[0, 1-a\right]^{\vert \overline{D}_{X} \vert} \times E_{Y} \right),
\end{equation} 
since we are looking for the $\underline{t} \in A$ such that $\underline{t}_{X} \notin E_{X}$ and  $\underline{\tau}_{Y} \notin E_{Y}$ simultaneously. \\

\begin{lemma}\label{dimE} Let $E$ be the set of parameters defined above. 
\begin{itemize}
\item[(a)]If all the defining parameters $a_i$, $b_j$ and vector $\underline{t}$ are algebraic, and the IFSs $\lbrace a_{i}x + t_{i}\rbrace_{ i \in \overline{D}_{X}}$ and $\lbrace b_{j} y+ \tau_{j}\rbrace_{ j \in \overline{D}_{Y}}$ do not have an exact overlap, then $\underline{t} \notin E$. 
\item [(b)] $\dim_H(E)=\dim_P(E)=\vert \overline{D}_{X} \vert + \vert \overline{D}_{Y} \vert -1.$
\end{itemize}
\end{lemma}

\begin{proof}
The first statement follows directly from the definition of $E$ and Proposition \ref{H2}. In order to prove property $(b)$, let ``$\dim$'' denote either the Hausdorff or packing dimension. Then we have $\dim F=n$ when $F$ is a hypercube of $\mathbb{R}^{n}$, and $\dim \left( F_{1} \cup F_2 \right)= \max \lbrace \dim F_{1}, \dim F_2 \rbrace.$ See for example \cite [Chapter 3]{FALbook}. Besides, it is shown in \cite{PSP} that the dimensions of the product of metric spaces satisfy
\begin{equation*}\label{prodd}
\dim_{H}E + \dim_{H}F \leq \dim_{H}(E \times F) \leq \dim_{H}E + \dim_{P}F \leq \dim_{P}(E \times F) \leq \dim_{P}E + \dim_{P}F.
\end{equation*}
By Proposition \ref{H2}, 
\begin{equation*}
\dim(E_{X})=|\overline{D}_{X}| - 1 \quad \qquad \dim (E_{Y})=|\overline{D}_{Y}| - 1,
\end{equation*}
\noindent Thus, putting all together we get
\begin{equation*}
\begin{split}
\dim_{H} E &= \max \Bigg\{ \dim_{H} \left( E_{X}\times \left[0, 1-b\right]^{\vert \overline{D}_{Y} \vert} \right), \dim_{H}\left( \left[0, 1-a\right]^{\vert \overline{D}_{X} \vert} \times E_{Y} \right) \Bigg\}\\
&= \max \big \{ \dim_{H} E_{X} + \vert \overline{D}_{Y} \vert,  \vert \overline{D}_{X} \vert + \dim_{H} E_{Y}\big \}\\
&= \vert \overline{D}_{X} \vert + \vert \overline{D}_{Y} \vert -1,
\end{split}
\end{equation*}
and
\vspace{-5pt}
\begin{equation*}
\dim_{H} E \leq \dim_{P} E \leq \max \big \{ \dim_{P} E_{X} + \vert \overline{D}_{Y} \vert, \dim_{P} E_{Y}+ \vert \overline{D}_{X} \vert\big \}
= \vert \overline{D}_{X} \vert + \vert \overline{D}_{Y} \vert -1.
\end{equation*}
\end{proof}

Using Theorem \ref{H1} and Lemma \ref{vitali}, we are now able to prove that for all parameters $\underline{t}$ outside $E$, the Bedford-McMullen-type systems $\mathcal{I}_k$ defined in equation (\ref{BMcarpet}) can be approximated by with subsystems with ``enough maps'' and without overlapping rows. With that aim let
\begin{equation*}
\overline{\Gamma}^{Y}_k=\big \{(j_1, \ldots, j_{\theta(k)}) : ((i_1,j_1), \ldots, (i_{\theta(k)},j_{\theta(k)})) \in \Gamma_k \text{ for some }i_1, \ldots i_{\theta(k)} \big \}, 
\end{equation*}
\noindent and for any fixed $\underline{t} \in A\setminus E$ consider the corresponding associated IFS of similarities 
\begin{equation*}
 \mathcal{I}^{Y}_k =\big \{ \overline{S}_{\underline{t},\mathbold{\lambda_{k}}}\big \}_{\mathbold{\lambda_{k}} \in \overline{\Gamma}_k^{Y}}.
\end{equation*}

\begin{lemma} \label{G21} Let $\underline{t} \in A \setminus E$ and $\mathcal{I}_k$ be the Bedford-McMullen-type system with uniform fibres and attractor $\Lambda_k$ defined in equation (\ref{BMcarpet}). For a given $\epsilon >0 $ there exists $\ell_{0} \in \mathbb{N}$  such that for all $\ell \geq \ell_{0} $ we can define a new system $\mathcal{L}_{\ell}=\lbrace S_{j} \rbrace_{j \in G_{k,\ell}}$ with attractor $\Upsilon_{\ell} \subseteq \Lambda_k$ and such that $\mathcal{L}_{\ell}^{Y}$ satisfies the OSC. Besides, $G_{k, \ell}\subseteq \Gamma_k^{\ell}$ and 
\begin{equation*}
\vert G_{k, \ell} \vert \geq 
3^{-1}(1/n_k)^{ \ell \epsilon} \vert \Gamma_{k}\vert^{\ell}. 
\end{equation*}
\end{lemma}
\vspace{7pt}
\begin{proof}
Let $\Lambda_k^{Y}$ be the attractor of the projected system $\mathcal{I}^{Y}_k $. Since  $\underline{t} \notin E$, by Theorem \ref{H1} we have that 
\begin{equation} \label{alphak}
\dim_{H}(\Lambda_k^{Y})=\frac{\log \vert \overline{\Gamma}_k^{Y}\vert}{\log n_k}=:\overline{s}_k^{Y},
\end{equation}
that satisfies $0\leq \overline{s}^{Y}_{k}\leq 1$. 
Using Lemma \ref{vitali}, we can approximate $\mathcal{I}^{Y}_{k}$ by a subsystem satisfying the SSC by assigning to the parameters of the mentioned lemma the values $\alpha= \overline{s}^{Y}_{k}$, $a=n_k^{-1}$. Then there exists $\ell_{0}\in \mathbb{N}$ so that  for $\ell \geq \ell_{0} $ we may find  
\[ \overline{G}^{Y}_{k,\ell} \subset (\overline{\Gamma}^{Y}_{k})^{\ell}\]
such that the system $ \lbrace \overline{S}_{i,\underline{t}}\rbrace_{i \in \overline{G}^{Y}_{k,\ell} } $  satisfies the SSC, and
\begin{equation}\label{boundss}
\vert \overline{G}^{Y}_{k,\ell} \vert \geq 3^{-\overline{s}^{Y}_{k}}(1/n_k)^{- \ell(\overline{s}^{Y}_{k}-\epsilon)}\geq 3^{-1}(1/n_k)^{ \ell \epsilon} \vert \overline{\Gamma}^{Y}_{k}\vert^{\ell},
\end{equation}
since by equation (\ref{alphak}) we have $\vert \overline{\Gamma}_{k}^{Y}\vert = n_k^{\overline{s}^{Y}_{k}}$ and $0\leq \overline{s}^{Y}_{k} \leq 1$. \\

\noindent We fix any such $\ell \geq \ell_{0}$ and define the set  
\[G_{k, \ell}:=\lbrace ((i_{1},j_{1}),\ldots ,(i_{\theta(k)\ell},j_{\theta(k)\ell}) \in \Gamma_k^{\ell}:(j_{1},\ldots ,j_{\theta(k)\ell}))  \in \overline{G}^{Y}_{k,\ell} \rbrace,\]
\noindent that is, we are considering the set of all rectangles of generation $\theta(k)\ell$ in $[0,1]^{2}$  whose projection onto the vertical axis belongs to $\overline{G}_{k,\ell}^{Y}$. Observe that the system $\mathcal{I}_k$ having uniform fibers means that $\vert \Gamma_k\vert =\vert  \overline{\Gamma}^{Y}_{k}\vert  J $, with $J=\vert \Gamma_k \vert/\vert\overline{\Gamma}^{Y}_{k}\vert $
being the number of elements in a row. Therefore, under iteration we get
\[ \vert \Gamma_k^{\ell} \vert  =\vert  \overline{\Gamma}^{Y}_{k}\vert^{\ell}  J ^{\ell} ,\]
with the rows of the IFS  $\lbrace S_{j} \rbrace_{j \in \Gamma_k^{\ell}}$ having $ J ^{\ell}$ elements. The set $G_{k, \ell}$ is a subset of $\Gamma_k^{\ell}$ obtained by removing some of its rows whilst keeping $\vert \overline{G}^{Y}_{k,\ell}\vert$ of them, so using equation (\ref{boundss}) we get
\begin{equation*} \label{G2}
\vert G_{k, \ell} \vert =\vert \overline{G}^{Y}_{k,\ell}\vert J ^{\ell}= \vert \overline{G}^{Y}_{k,\ell}\vert\left( \frac{\vert \Gamma_{k}\vert }{\vert \overline{\Gamma}_{k}^{Y} \vert }\right)^{\ell} \geq 
3^{-1}(1/n_k)^{ \ell \epsilon} \vert \Gamma_{k}\vert^{\ell}. 
\end{equation*} 
\end{proof}

\noindent We proceed now to present the argument to get a lower bound for the Hausdorff dimension. The idea is to apply Fraser-Shmerkin's Theorem to the attractors $\Upsilon_{\ell}$. This theorem provides us with a formula for their dimension in terms of the number of rectangles on each column. By the previous lemma, the the total number of maps that generate $\Upsilon_{\ell}$ is ``high enough'' to prove that $\dim_{H}(\Upsilon_{\ell})$ doesn't drop from our target dimension even for the configuration in columns that gives the smallest Hausdorff dimension.

\begin{prop} \label{lowerH}
Let $\underline{t} \in A \setminus E$. Then
\[ \dim_{H}(F_{\underline{t}})\geq \max_{\textbf{p}\in\mathbb{P}^{|D|}}g(\textbf{p}). \]
\end{prop}
\begin{proof}

Let $\mathcal{I}_k$ be the IFS defined in (\ref{BMcarpet}), to which we apply Lemma \ref{G21} to get a new system $\mathcal{L}_{\ell}=\lbrace S_{j} \rbrace_{j \in G_{k,\ell}}$  free of overlapping rows and with attractor $\Upsilon_{\ell}$. Note that by the choice of $\underline{t} \notin E$, the projected system $\lbrace \overline{S}_{i, \underline{t}}\rbrace_{i \in \overline{D}_{X}}$ does not have super exponential concentration of cylinders, and thus by Remark \ref{tilde}, neither does $\overline{\mathcal{I}}^{X}_\ell$. Consequently, we can apply Fraser-Shmerkin's Theorem \ref{FSthm} to $\Upsilon_{\ell}$ and get that 
\begin{equation*}\label{formF}
 \dim_{H}(\Upsilon_{\ell})=\dfrac{\log \sum_{i \in \overline{G}^{X}_{k,\ell}} \vert I_{i}\vert^{\frac{\log m_{k}^{\ell}}{\log n_{k}^{\ell}}}}{\log m_{k}^{\ell}},
\end{equation*}
with $\vert I_{i}\vert$ being the number of chosen rectangles in the $i$th column of $\mathcal{L}_{\ell}$. After this last subsystem approximation performed, we don't know the exact values that the variables $\vert I_{i}\vert$ take. Nonetheless, since $\mathcal{I}_k$ has uniform fibers and  by Lemma \ref{G21}, we have that $G_{k, \ell}\subseteq \Gamma^{\ell}$, in addition to the bounds
\begin{equation}\label{revH}
\vert I_{i}\vert\leq I^{\ell}=\frac{\vert \Gamma_k\vert^{\ell}}{\vert \overline{\Gamma}^{X}_{k}\vert^{\ell}}\quad  \text{ for each } i \in \overline{G}^{X}_{k,\ell} \qquad  \text{ and }  \qquad \sum_{i \in \overline{G}^{X}_{k,\ell}} \vert I_{i}\vert=\vert G_{k, \ell} \vert \geq 
3^{-1}(1/n_k)^{ \ell \epsilon} \vert \Gamma_{k}\vert^{\ell}. 
\end{equation}

Let $\gamma=\frac{\log m_k}{\log n_k}$, $N=I^{\ell}$ and $T=\vert G_{k, \ell} \vert$. We shall see that $\dim_{H}(\Upsilon_{\ell})\geq s_k$, with $s_k$ as in (\ref{s_k}), for any distribution in columns of $T$ rectangles. In particular it suffices to prove it for the distribution that minimizes $\dim_{H}(\Upsilon_{\ell})$. Thus, we start by addressing the optimization problem of finding integers $0\leq N_i\leq N$ minimizing $\sum_{i} N_i^{\gamma}$ and such that $\sum N_i\geq T$.\\ 

Observe that since by equation (\ref{pinB}) we have $0<\gamma<1$, for $0< N_i\leq N_j$ the functions $f_{N_i,N_j}(x)=(N_i-x)^{\gamma} +(N_j+x)^{\gamma}$ are decreasing for $0<x\leq N_i$. In particular $(N_i-1)^{\gamma} +(N_j+1)^{\gamma}<N_i^{\gamma}+N_j^{\gamma}$, which means that there cannot be  two elements $0< N_i\leq N_j<N$ as part of the solution. Otherwise they could be replaced by the pair $(N_i-1, N_j+1)$, contradicting the minimality of the objective function. Thus, the minimum is attained at  
$$ \Big( \underbrace{N, \ldots, N}_{ \left\lfloor T/N \right \rfloor \text{ times }},\quad (T-\left\lfloor T/N \right \rfloor N), \underbrace{0, \ldots\ldots ,0}_{  \vert \overline{G}^{X}_{k,\ell} \vert -  \left\lfloor T/N \right \rfloor-1\text{ times }}  \Big),$$ 
\noindent for which the objective  $\sum_{i} N_i^{\gamma}$ function takes the value
\begin{equation*} \label{optim}
\left \lfloor \frac{T}{N} \right \rfloor N^{\gamma}+ \left(T-\left \lfloor \frac{T}{N} \right \rfloor N\right)^{\gamma}.
\end{equation*}

\noindent Therefore, for the original distribution $ \lbrace \vert I_i\vert\rbrace_i$ on $\mathcal{L}_\ell$, we get the bound $
\sum_{i \in \overline{G}^{X}_{k,\ell}} \vert I_i\vert^{\gamma} \geq \left \lfloor \frac{T}{N} \right \rfloor N^{\gamma}+ \left(T-\left \lfloor \frac{T}{N} \right \rfloor N\right)^{\gamma} \geq \left \lfloor \frac{T}{N} \right \rfloor N^{\gamma}
$, which together with equation (\ref{revH}) gives
\begin{equation*}\label{lim2}
\begin{split}
\dim_{H}(\Upsilon_{\ell}) & \geq \dfrac{\log (\lfloor\vert G_{k, \ell} \vert/I^{\ell} \rfloor I^{\ell \gamma})}{\log m_{k}^{\ell}}\geq \dfrac{\log \vert G_{k, \ell}\vert I^{\ell(\gamma-1)}}{\log m_{k}^{\ell}}- c_\ell\\
&= \frac{\log |G_{k,\ell}|}{ \ell \log m_k}+ \frac{(\gamma-1)\log (\vert \Gamma_k\vert/\vert \overline{\Gamma}^{X}_{k}\vert)}{ \log m_k}- c_\ell\\
&\geq \frac{\log \left( 3^{-1}(1/n_k)^{\ell \epsilon} \vert \Gamma_k\vert^{\ell}\right)}{\ell \log m_k}+ \frac{(\gamma-1)\log (\vert \Gamma_k\vert/\vert \overline{\Gamma}^{X}_{k}\vert)}{ \log m_k}-c_{\ell}\\
&= \frac{\log \vert\overline{\Gamma
}^{X}_{k} \vert}{ \log m_k}+ \frac{\log \vert \Gamma_k \vert-\log \vert\overline{\Gamma
}^{X}_{k} \vert}{ \log n_k}-c_{\ell} -\gamma^{-1}\epsilon -\zeta_{\ell}\\
&=s_{k}-c_{\ell} -\gamma^{-1}\epsilon -\zeta_{\ell},
\end{split}
\end{equation*}
where the auxiliary variables $c_{\ell}=\frac{1}{\ell \log m_{k}}$ and $\zeta_{\ell}=\frac{\log 3} { \ell \log m_k}$ converge to $0$ when $\ell \rightarrow \infty$.  Thus, letting $\epsilon$ tend to zero, using Lemma \ref{sktogp} and by monotonicity of the Hausdorff dimension we get the desired lower bound.
\end{proof}

\subsection{Calculation of the dimension}
We can now complete the proof of our results:

\begin{proof}[Proof of Theorem \ref{THMhaus}]
It follows directly from Propositions \ref{upperH}, \ref{lowerH} and Lemma \ref{dimE}.
\end{proof}

\begin{proof}[Proof of Corollary \ref{corHauss}]
By Fraser-Shmerkin's Theorem \ref{FSthm},  $\dim_{H}(F)=\frac{\log \sum_{i \in \overline{D}_{X}}\vert I_{i} \vert^{\frac{\log \tilde{m}}{\log \tilde{n}}}}{\log \tilde{m}}$, or in other words, for this particular case of Bara\'{n}ski carpets, $\max_{\textbf{q}\in\mathbb{P}^{|D|}}g(\textbf{q})$ is reached for the weights vector with coordinates $p_{ij}= \vert I_{i} \vert^{\frac{\log \tilde{m}}{\log \tilde{n}}-1}/\tilde{m}^{s}. $ Thus, by  Theorem \ref{THMhaus}, this will also be the Hausdorff dimension of $F_{\underline{t}}$ for all parameters $\underline{t}$ outside the set $E$. Let us now consider the hyperplane $$\mathcal{P}=\lbrace \underline{t} \in A : t_{i_{1}}=t_{i_{2}} \text{ for some }i_1,i_2 \in \overline{D}_{X} \rbrace.$$ 

This merges two columns of our original pattern, i.e. we have an exact overlap, and symbolically this is equivalent to replacing two columns with $N_{i_1}$, $N_{i_2}$ rectangles by a single column with $N\leq N_{i_1} + N_{i_2}$ rectangles. Since $(N_{i_1} + N_{i_2})^{\gamma} < N_{i_1}^{\gamma} + N_{i_2}^{\gamma} $ for any $\gamma \in (0,1)$ and in particular for $\gamma=\frac{\log \tilde{m}}{\log \tilde{n}}$, we have
\begin{equation*}
\dim_H (F_{\underline{t}})\leq \frac{\log (N_{i_1} + N_{i_2})^{\frac{\log \tilde{m}}{\log \tilde{n}}} + \sum_{\lbrace i=1,\ldots, m \rbrace \setminus \lbrace i_1, i_2 \rbrace} N_i^{\frac{\log \tilde{m}}{\log \tilde{n}}}}{\log \tilde{m}}<\dim_{H}(F).
\end{equation*}
\vspace{1pt}

\noindent Thus, $\mathcal{P} \subseteq E_0$, and since $\dim_{H}\mathcal{P}=\vert\overline{D}_{X}\vert + \overline{D}_{Y}\vert-1$, $\dim_{H} E_0 \geq \vert\overline{D}_{X}\vert + \vert  \overline{D}_{Y}\vert-1$. Note that $E_0\subseteq E$, although these sets are not necessarily equal. But this inclusion implies $\dim_{H}E_0\leq \dim_{H}E= \vert \overline{D}_{Y}\vert + \vert \overline{D}_{Y}\vert-1$ by Lemma \ref{dimE} $(b)$. Hence, by the sandwich lemma $\dim_{H}E_0=\vert\overline{D}_{X}\vert + \vert \overline{D}_{Y}\vert-1$. The same argument applies to the packing dimension of $E_0$, which concludes the proof.  
\end{proof}

\section{Calculation of the box-counting and packing dimension}\label{sectionBoxC}

This section is devoted to the proof of Theorem \ref{THMbox} and Corollary \ref{corBox}. We include here the box dimension case and note that the same result  is true for the packing dimension. This is due to the fact that each $F_{\underline{t}}$ is a compact set for which every open ball centred at it contains a bi-Lipschitz image of $F_{\underline{t}}$. Therefore, we can conclude that  $\dim_P F_{\underline{t}} = \dim_B F_{\underline{t}}$ for all $\underline{t}$. For more details see \cite[Corollary 3.10]{FALbook}. We now recall the definition of box-counting dimension. \\

\begin{defn}\label{boxx} Let $F$ be a non-empty bounded subset of $\mathbb{R}^{n}$. A $\delta$-\textit{cover} of $F$ is a collection of sets $\lbrace U_{i} \rbrace_{i=1}^{\infty}$ in $\mathbb{R}^{n}$ such that $F \subset \bigcup_{i=1}^{\infty} U_{i}$ and $\text{diam}(U_{i})\leq \delta$ for each $i$. Let $N_{\delta}(F)$ be the least number of sets in any $\delta$-cover of $F$. Then the \textit{lower} and \textit{upper box-counting dimensions} of $F$ are defined as 
 \begin{equation} \label{defb}
 \underline{\dim}_{B}F=\liminf_{\delta \rightarrow 0} \frac{\log N_{\delta}(F)}{-\log \delta}, \qquad \overline{\dim}_{B}F=\limsup_{\delta \rightarrow 0} \frac{\log N_{\delta}(F)}{-\log \delta}
 \end{equation}
respectively. If both limits are equal, we refer to the common value as the \textbf{box-counting dimension} of $F$:
\begin{equation}\label{defbb}
\dim_{B}F=\lim_{\delta \rightarrow 0} \frac{\log N_{\delta}(F)}{-\log \delta}.
\end{equation}
\end{defn}
We remark that $N_{\delta}$ can adopt several definitions all based on covering or packing the set at scale $\delta$, see \cite [Section 3.1]{FALbook}. In particular, for us $N_{\delta}$ will denote the number of cubes in an $\delta$-grid which intersect $F$. The next two properties, that follow directly from Definition \ref{boxx}, will help us finding the box dimension of our setting:

\begin{prop} \label{propbox} The following hold:
\vspace{-5pt}
\begin{enumerate}
\item In (\ref{defb}) and (\ref{defbb}) it is enough to consider limits as $\delta$ tends to $0$ through any decreasing sequence $\delta_{\ell}\rightarrow 0$ such that $\delta_{\ell+1} \geq a \delta_{\ell}$ for some constant $0 < a < 1$. In particular 
\begin{equation} \label{eqBox}
\liminf_{\delta_{\ell} \rightarrow 0} \frac{\log N_{\delta_{\ell}}(F)}{-\log \delta_{\ell}}\leq \liminf_{\delta \rightarrow 0} \frac{\log N_{\delta}(F)}{-\log \delta}.
\end{equation}

\item Let $s=\dim_{B}F$ and  let $\epsilon>0$. Then there exists a constant $C_{\epsilon} > 0$ such that for all $\delta \in (0,1]$
\begin{equation*}\label{Nd}
N_{\delta}(F) \geq C_{\epsilon} \delta^{-(s-\epsilon)}.
\end{equation*}
\end{enumerate}
 \end{prop}

\begin{proof}
\begin{enumerate}
\item For any $\delta\in (0,1]$ there exists $\ell$ such that  $\delta_{\ell +1} \leq \delta \leq \delta_{\ell}$, and thus $-1/\log \delta_{\ell +1} \leq - 1/\log \delta $ and $ \log N_{\delta}(F)\geq \log N_{\delta_{\ell}}(F)$.
Therefore,
 \begin{equation*}
 \begin{split}
\frac{ \log N_{\delta_{\ell}}(F)}{-\log \delta_{\ell}} &\leq \frac{\log N_{\delta}(F) }{-\log \delta_{\ell+1} + \log (\frac{\delta_{\ell+1}}{\delta_{\ell}})} 
\leq \frac{\log N_{\delta}(F) }{-\log \delta_{\ell+1} + \log (a)} 
\leq \frac{\log N_{\delta}(F) }{-\log \delta + \log (a)}
=\frac{\frac{\log N_{\delta}(F)}{-\log \delta }}{1 + \frac{\log (a)}{-\log \delta}},
\end{split}
\end{equation*}
and taking lower limits we get equation (\ref{eqBox}). The opposite inequality is immediate, and the case of upper limits can be dealt with in the same way.

\item Given $\epsilon >0$, by Definition \ref{boxx} there exists $\delta_{\epsilon}$  such that $\frac{\log N_{\delta}(F)}{-\log \delta} \geq s-\epsilon$ for all $\delta \leq \delta_{\epsilon}$, or equivalently, by monotonicity of the logarithmic function, $N_{\delta}(F) \geq  \delta^{-(s-\epsilon)}$.  Also $\frac{\log N_{\delta}(F)}{-\log \delta}$ is a continuous function on $\delta>0$. Thus, by the Weierstrass extreme value theorem, it reaches a minimum value $m_{\epsilon}$ on the interval $[\delta_{\epsilon}, 1]$, so that $\frac{\log N_{\delta}(F)}{-\log \delta} \geq m_{\epsilon}$ for all $\delta \in [\delta_{\epsilon}, 1]$. Thus, if we choose $C_{\epsilon}=\min\lbrace 1, \mathrm{e}^{m_{\epsilon}}\rbrace$, we get the desired inequality for all $\delta \in (0,1]$. 
\end{enumerate}
\end{proof} 

\subsection{Upper bound}

The upper bound for the box-counting dimension comes as a direct consequence of Fraser's paper \cite{Boxlike} on the dimensions of a class of self-affine carpets including our attractors $F_{\underline{t}}$. As we shall see, no separation conditions are required for his result, and the dimensions will come in terms of the box-dimensions of the orthogonal projections of the systems studied. For the sake of clarity we have decided to follow this approach, but we remark that alternatively, Bara\'{n}ski's argument for the upper bound of the box dimension in \cite[Section 6]{Baranski} should adapt to our setting.\\   

\noindent Let $\underline{t}\in A$. Then for any sequence $\mathbold{\lambda_{k}} \in D^{k}$, according to whether it is an $A$-sequence or $B$-sequence, we define
\begin{equation*}
   \overline{s}_{\mathbold{\lambda_{k}}}=\left\lbrace
  \begin{array}{l}
    \dim_B \pi_{\text{\tiny X}}(F_{\underline{t}}) \quad \text{ if }  L_{\mathbold{\lambda_{k}}}= A_{\mathbold{\lambda_{k}}}
      \\[15pt]
       
    \dim_B \pi_{\text{\tiny Y}}(F_{\underline{t}}) \quad \text{ if }  L_{\mathbold{\lambda_{k}}}= B_{\mathbold{\lambda_{k}}} \\
  \end{array}.
  \right.\\
\end{equation*}
\medskip
\begin{rem}\label{remUpB}
Observe that the box dimension of each of the attractors of the projected IFSs $\lbrace \overline{S}_{\underline{t},i}\rbrace_{i \in \overline{D}_{X}}$ and $\lbrace \overline{S}_{\underline{t},j}\rbrace_{j \in \overline{D}_{Y}}$ is trivially bounded by its similarity dimension, that by equation (\ref{TB}) equals $t_A$ and $t_B$ respectively.
\end{rem}
\noindent Define
\begin{equation*}
\Psi_{\underline{t},k}^{s}= \sum_{\mathbold{\lambda_{k}} \in \mathcal{D}^{k}} L_{\mathbold{\lambda_{k}}}^{ \overline{s}_{\mathbold{\lambda_{k}}}} \, \,  T_{\mathbold{\lambda_{k}}}^{s-\overline{s}_{\mathbold{\lambda_{k}}}}.
\end{equation*}\\
The main result of \cite{Boxlike}, in an adapted version to our setting, states:
\begin{thm} \cite [Theorem 2.4]{Boxlike} \label{main}
 For each $\underline{t}\in A$ we have $\dim_\text{\emph{P}} F_{\underline{t}} = \overline{\dim}_\text{\emph{B}} F_{\underline{t}} \leq s$, where $s \geq 0$ is the unique solution of $P_{\underline{t}}(s) = 1$, and the function $P_{\underline{t}}$ is defined as \[
P_{\underline{t}}(s) := \lim_{k \to \infty} \left( \Psi_{\underline{t},k}^{s} \right)^{1/k}.
\]
\end{thm}

\begin{cor}\label{upperB}
For any $\underline{t} \in A $ we have 
\[ \dim_{B}(F_{\underline{t}})\leq \max (D_{A}, D_{B}),\]
where $D_{A}$, $D_{B}$ are the unique real numbers given by equation (\ref{DADB}).
\end{cor}

\begin{proof}
For each $k \geq 1$ we consider the partition of $D^{k}$ into the sets of $A$-sequences and $B$-sequences:
\begin{equation*}
\mathcal{A}_k=\big \lbrace \mathbold{\lambda_{k}} \in D^{k}   \text{ such that }   A_{\mathbold{\lambda_{k}}}\geq B_{\mathbold{\lambda_{k}}} \big \rbrace \qquad \mathcal{B}_{k}=D^{k} \setminus \mathcal{A}_k.\\
\end{equation*}

\noindent Then, by Remark \ref{remUpB}, for any $\underline{t} \in A $ the function $P_{\underline{t}}$ can be written as and bounded by 
\begin{equation}\label{boundP}
\begin{split}
P_{\underline{t}}(s) &=\lim_{k \to \infty} \left( \sum_{\mathbold{\lambda_{k}} \in \mathcal{A}_k} A_{\mathbold{\lambda_{k}}}^{\overline{s}_{\mathbold{\lambda_{k}}}} \, \,  B_{\mathbold{\lambda_{k}}}^{s-\overline{s}_{\mathbold{\lambda_{k}}}} + \sum_{\mathbold{\lambda_{k}} \in \mathcal{B}_k} B_{\mathbold{\lambda_{k}}}^{\overline{s}_{\mathbold{\lambda_{k}}}} \, \,  A_{\mathbold{\lambda_{k}}}^{s-\overline{s}_{\mathbold{\lambda_{k}}}} \right)^{1/k}\\
& \leq \lim_{k \to \infty} \left( \sum_{\mathbold{\lambda_{k}} \in \mathcal{A}_k} A_{\mathbold{\lambda_{k}}}^{t_A} \, \,  B_{\mathbold{\lambda_{k}}}^{s-t_A} + \sum_{\mathbold{\lambda_{k}} \in \mathcal{B}_k} B_{\mathbold{\lambda_{k}}}^{t_B} \, \,  A_{\mathbold{\lambda_{k}}}^{s-t_B} \right)^{1/k}.
\end{split}
\end{equation}

\noindent Note that by equations (\ref{TB}) it holds
\begin{equation}\label{minus1}
\sum_{(i,j)\in D} a_{i}^{t_{A}}b_{j}^{t_{B}}=\sum_{i \in \overline{D}_{X}} a_{i}^{t_{A}} \sum_{(i,j)\in I_i} b_{j}^{t_{B}} \leq \left(\sum_{i \in \overline{D}_{X}} a_{i}^{t_{A}} \right)\left(\sum_{j \in \overline{D}_{Y} } b_{j}^{t_{B}} \right)=1,  
\end{equation}
which by (\ref{DADB}) implies $D_A \leq t_A+ t_B $. The same reasoning applies to $D_B$, and therefore we have
\begin{equation} \label{desig}
\max(D_A, D_B)\leq t_A+ t_B.
\end{equation}

\noindent If we define $$\alpha^{s}_k := \max \left\{ \sum_{\mathbold{\lambda_{k}} \in D^{k}} A_{\mathbold{\lambda_{k}}}^{t_A} \, \,  B_{\mathbold{\lambda_{k}}}^{s-t_A},  \sum_{\mathbold{\lambda_{k}} \in D^{k}} B_{\mathbold{\lambda_{k}}}^{t_B} \, \,  A_{\mathbold{\lambda_{k}}}^{s-t_B} \right\},$$ then for all $s\leq t_A+t_B$ it holds

\begin{equation} \label{eqq}
\alpha^{s}_k \leq \sum_{\mathbold{\lambda_{k}} \in \mathcal{A}_k} A_{\mathbold{\lambda_{k}}}^{t_A} \, \,  B_{\mathbold{\lambda_{k}}}^{s-t_A} + \sum_{\mathbold{\lambda_{k}} \in \mathcal{B}_k} B_{\mathbold{\lambda_{k}}}^{t_B} \, \,  A_{\mathbold{\lambda_{k}}}^{s-t_B} \leq 2\alpha^{s}_k,
\end{equation}
where the second inequality is obvious as all terms of  both sums are positive, while the first one follows from  
\begin{equation*}
 A_{\mathbold{\lambda_{k}}}^{t_A} \, \,  B_{\mathbold{\lambda_{k}}}^{s-t_A} \leq B_{\mathbold{\lambda_{k}}}^{t_B} \, \,  A_{\mathbold{\lambda_{k}}}^{s-t_B} \,\Longleftrightarrow \, B_{\mathbold{\lambda_{k}}}^{s-t_A-t_B} \leq   A_{\mathbold{\lambda_{k}}}^{s-t_B-t_A}\, \Longleftrightarrow\, B_{\mathbold{\lambda_{k}}}\geq A_{\mathbold{\lambda_{k}}} \, \Longleftrightarrow\, \mathbold{\lambda_{k}} \in \mathcal{B}_{k}.    
\end{equation*}\\
\noindent Furthermore, it is easy to see that 
$$\sum_{\mathbold{\lambda_{k}} \in D^{k}} A_{\mathbold{\lambda_{k}}}^{t_A} \, \,  B_{\mathbold{\lambda_{k}}}^{s-t_A}= \left( \sum_{(i,j) \in D} a_i^{t_A} \, \,  b_{j}^{s-t_A} \right)^{k} \,\quad \text{and} \,\quad \sum_{\mathbold{\lambda_{k}} \in D^{k}} B_{\mathbold{\lambda_{k}}}^{t_B} \, \,  A_{\mathbold{\lambda_{k}}}^{s-t_B}= \left(  \sum_{(i,j) \in D }  b_j^{t_B} \, \,  a_{i}^{s-t_B} \right)^{k},$$
as $D^{k}$ comprises all possible combinations of length $k$ of elements in $D$, and hence
$$\alpha^{s}_k =  \left( \max \left\{\sum_{(i,j) \in D} a_i^{t_A} \, \,  b_{j}^{s-t_A}, \sum_{(i,j) \in D }  b_j^{t_B} \, \,  a_{i}^{s-t_B}  \right\}\right)^{k}.$$ \\

\noindent Thus, by this and equation (\ref{DADB}), $\lim_{k \to \infty} (\alpha_k^{s})^{1/k}=1$ when $s= \max(D_A, D_B) $. Then by equations (\ref{desig}), (\ref{eqq}) and (\ref{boundP}), and since $P_{\underline{t}}(s)$ is strictly decreasing on $[0,\infty)$ (see \cite[Lemma 2.2]{Boxlike}), $P_{\underline{t}}(\tilde{s})=1$ for some $\tilde{s}\leq \max(D_A, D_B)$, and the result follows from Theorem \ref{main}. 

\end{proof}

\subsection{Lower bound}
Let $F$ be a fixed Bara\'{n}ski carpet. We can assume without loss of generality that 
\begin{equation}\label{max}
\dim_B F=\max (D_{A}, D_{B})=D_A.\\
\end{equation}

We will follow a similar argument to that in Section \ref{sectionLower}. We start by approximating each system in our parametric family of carpets by a sequence $\mathcal{I}_k$ of (possibly overlapping) Bedford-McMullen-type systems of parameters $(m_k, n_k)$ with uniform fibres, this time using a different vector of weights. In order to perform the right further approximations, we need to show that $\max (D_{A}, D_{B})=D_A$ implies $m_k\geq n_k$. For that purpose and ispired by \cite[Section 6]{Baranski},  we will use $\delta$-covers of $D^{\mathbb{N}}$, induce Bernoulli measures on them, and show that the measures will be mostly concentrated on the level sets for which $A_{\mathbold{\lambda_{M_\delta}}}\geq B_{\mathbold{\lambda_{M_\delta}}}$. Then the result will follow in Lemma \ref{ngeqm}.\\

Again, to each $\mathcal{I}_k$ we will associate a number $s_k$ and prove in Lemma \ref{sktoDB} that these are increasingly good approximations of $\dim_B(F)$. Then the key step towards the proof of Theorem \ref{THMhaus} is Lemma \ref{lemmaB}. It provides us, for any translation vector outside the $E$ defined in (\ref{setE}), with subsystems satisfying the OSC defined as approximations of iterations of the systems $\mathcal{I}_k$ . These subsystems will have  ``enough maps'' as to prove in Proposition \ref{lowerB} that the $s_k$ were also lower bounds of $\dim_B(F_{\underline{t}})$ for all $\underline{t} \in A\setminus E$.\\

\noindent Let $\textbf{p}\in\mathbb{P}^{|D|}$ with coordinates 
\begin{equation*}
p_{i j}=a_{i}^{t_{A}}b_{j}^{D_{A}-t_{A}}.
\end{equation*}\\
\noindent For $k\in\mathbb{N}$, set $\theta(k) = \sum_{(i,j)\in D} \lceil k p_{i j} \rceil$, and define
\begin{equation*}
 \Gamma_k= \Bigg\{
 \begin{tabular}{c}
 $\mathbold{\lambda_{k}}=(\lambda_{1},\lambda_{2},\ldots ,\lambda_{\theta(k)})\in D^{\theta(k)} : \text{ for all } (i,j) \in D,$ \\
 $ |\lbrace l \in \lbrace 1, \ldots,\theta(k) \rbrace : \lambda_{l}=(i,j) \rbrace|= \lceil k p_{i j}\rceil  $
 \end{tabular}
 \Bigg \}.
\end{equation*}
For each $\underline{t}\in A$ the set $\Gamma_k$ defines an IFS 
\begin{equation}\label{dbcarpet2}
\mathcal{I}_k:=\big \{ S_{\underline{t},\mathbold{\lambda_{k}}}\big \}_{\mathbold{\lambda_{k}} \in \Gamma_k}
\end{equation}
with uniform fibres. Let us denote the attractor associated to $\mathcal{I}_k$ by  $\Lambda_k$. By construction, $\Lambda_k \subset F_{\underline{t}}$, and the linear part of each map on $\mathcal{I}_k$ is given by
\begin{equation}\label{mn}
 \text{diag}\left( \prod_{(i,j)\in D} a_{i}^{\lceil k p_{i j} \rceil},\prod_{(i,j)\in D} b_{j}^{ \lceil k p_{i j} \rceil} \right)=:\text{diag}(m_k^{-1},n_k^{-1}).
\end{equation}

Since $m_k$ and $n_k$ are not necessarily integers, we have that $\mathcal{I}_k$ generates a Bedford-McMullen-type carpet. We shall see now, using ideas from \cite[Section 6]{Baranski}, that the assumption $D_A \geq D_B$ made in (\ref{max}) implies $m_k \leq n_k.$ \\

Observe that after some iteration, the $k$-level sets of our construction will be rectangles of different sizes. In order to get some control on the length of their shorter edge, we define a cover of $D^{\mathbb{N}}$ by cylinders of different levels $l$ such that $T_{\mathbold{\lambda_{l}}} \leq \delta$ for $\mathbold{\lambda}\in C_{\mathbold{\lambda_{l}}}$ and any fixed $\delta>0$. \\

\begin{defn} For a fixed $\delta>0$ and  each $\mathbold{\lambda}  \in D^{\mathbb{N}}$ define 
\[ M_{\delta}= M_{\delta}(\mathbold{\lambda})=\min \lbrace l: T_{\mathbold{\lambda_{l}}} \leq \delta \rbrace, \qquad \qquad  \mathcal{V}_\delta =\Big\{ \mathcal{C}_{\mathbold{\lambda_{M_\delta}}} : \mathbold{\lambda}\in D^{\mathbb{N} }\Big\} \]

\[\mathcal{V}_{\delta}^{(A)}=\big \{ \mathcal{C}_{\mathbold{\lambda_{M_\delta}} }\in \mathcal{V}_\delta : 2A_{\mathbold{\lambda_{M_\delta}}}\geq  B_{\mathbold{\lambda_{M_\delta}}} \big \} \qquad \mathcal{V}_{\delta}^{(B)}=\mathcal{V}_{\delta} \setminus \mathcal{V}_{\delta}^{(A)}. \]
\end{defn}

\begin{rem} \label{remcover}
$\mathcal{V}_\delta$ is a cover of $ D^{\mathbb{N}}$ consisting of pairwise disjoint sets. Being a cover follows from its definition, since $\mathcal{V}_\delta$ contains the cylinders associated to all $\mathbold{\lambda}\in D^{\mathbb{N}}$. To see that the sets are disjoint suppose $\mathbold{\gamma}\in  \mathcal{C}_{\mathbold{\lambda_{M_\delta}} } \cap \mathcal{C}_{\mathbold{\lambda'_{M'_\delta}} }$ and assume $M_\delta \leq M'_\delta$. Then it must occur $\mathbold{\gamma_{M_\delta}}=\mathbold{\lambda_{M_\delta}}=\mathbold{\lambda'_{M_\delta}}$, but by definition of $M'_\delta$ we have $M_\delta=M'_\delta$. \\
\end{rem}

\noindent By definition of $D_A$ we have
$$\sum_{\mathbold{\lambda_1}=(i_1,j_1)}^{(i_d,j_d)} A_{\mathbold{\lambda_1}}^{t_{A}}B_{\mathbold{\lambda_1}}^{D_{A}-t_{A}}=1,$$which implies that for all $k\geq 1$ and all $\mathbold{\lambda_{k}} \in D^k$
\begin{equation} \label{rel1}
\sum_{\lambda_{k+1}=(i_1,j_1)}^{(i_d,j_d)} A_{\mathbold{\lambda_{k+1}}}^{t_{A}}B_{\mathbold{\lambda_{k+1}}}^{D_{A}-t_{A}}= A_{\mathbold{\lambda_{k}}}^{t_{A}}B_{\mathbold{\lambda_{k}}}^{D_{A}-t_{A}},
\end{equation} 
where $\lambda_{k+1}$ is the last coordinate of the vector $\mathbold{\lambda_{k+1}}=(\lambda_{1},\ldots,\lambda_{k+1})$.  Observe that if $M_\delta(\tilde{\mathbold{\lambda}})=\max_{\mathbold{\lambda} \in D^\mathbb{N}}\lbrace M_\delta(\mathbold{\lambda})\rbrace$, the cylinders $[\tilde{\mathbold{\lambda}}_{M_\delta -1}, \lambda']$ must also belong to $\mathcal{V}_\delta$ for all $\lambda'\neq \lambda_{M_\delta}$, and then we can iteratively apply the previous relation (\ref{rel1}) to get 
\begin{equation*} \label{vdelta}
\sum_{\mathcal{C}_{\mathbold{\lambda_{M_\delta}} }\in \mathcal{V}_\delta} A_{\mathbold{\lambda_{M_\delta}}}^{t_{A}}B_{\mathbold{\lambda_{M_\delta}}}^{D_{A}-t_{A}}=1.
\end{equation*}
Analogously, 
\begin{equation*}
\sum_{\mathcal{C}_{\mathbold{\lambda_{M_\delta}}} \in \mathcal{V}_\delta} B_{\mathbold{\lambda_{M_\delta}}}^{t_{B}}A_{\mathbold{\lambda_{M_\delta}}}^{D_{B}-t_{B}}=1.
\end{equation*}

\noindent We can rewrite these equations in terms of the partition $\lbrace \mathcal{V}_\delta^{(A)}, \mathcal{V}_\delta^{(B)} \rbrace$  of   $ \mathcal{V}_\delta$:
\begin{equation}  \label{vdelta2}
  \begin{aligned}
\sum_{\mathcal{C}_{\mathbold{\lambda_{M_\delta}}}\in \mathcal{V}_\delta^{(A)}} A_{\mathbold{\lambda_{M_\delta}}}^{t_{A}}B_{\mathbold{\lambda_{M_\delta}}}^{D_{A}-t_{A}} + \sum_{\mathcal{C}_{\mathbold{\lambda_{M_\delta}}}\in \mathcal{V}_\delta^{(B)}} A_{\mathbold{\lambda_{M_\delta}}}^{t_{A}}B_{\mathbold{\lambda_{M_\delta}}}^{D_{A}-t_{A}}=1, \\[15pt]
\sum_{\mathcal{C}_{\mathbold{\lambda_{M_\delta}}} \in \mathcal{V}_\delta^{(A)}} B_{\mathbold{\lambda_{M_\delta}}}^{t_{B}}A_{\mathbold{\lambda_{M_\delta}}}^{D_{B}-t_{B}}+ \sum_{\mathcal{C}_{\mathbold{\lambda_{M_\delta}}} \in \mathcal{V}_\delta^{(B)}} B_{\mathbold{\lambda_{M_\delta}}}^{t_{B}}A_{\mathbold{\lambda_{M_\delta}}}^{D_{B}-t_{B}}=1.
  \end{aligned}
\end{equation} 

We will assume from now on that $t_A+t_B >D_A \geq D_B$, and will deal with the easier case of $t_A+t_B=D_A$ at the end of the proof of Proposition \ref{lowerB}. By equation (\ref{desig}) and the definition of $M_\delta$, for every $\mathcal{C}_{\mathbold{\lambda_{M_\delta}}} \in \mathcal{V}_\delta^{(B)}$ we get
\begin{equation*} \label{eqcomparison}
\frac{A_{\mathbold{\lambda_{M_\delta}}}^{t_{A}}B_{\mathbold{\lambda_{M_\delta}}}^{D_{A}-t_{A}}}{B_{\mathbold{\lambda_{M_\delta}}}^{t_{B}}A_{\mathbold{\lambda_{M_\delta}}}^{D_{B}-t_{B}}}=\frac{A_{\mathbold{\lambda_{M_\delta}}}^{t_{A} +t_{B}-D_B}}{B_{\mathbold{\lambda_{M_\delta}}}^{t_{B} +t_{A}-D_A}}=\left(\frac{A_{\mathbold{\lambda_{M_\delta}}}}{B_{\mathbold{\lambda_{M_\delta}}}}\right)^{t_{A} +t_{B}-D_A} \!\!\!\! \! A_{\mathbold{\lambda_{M_\delta}}}^{D_A-D_B}< \left(\frac{1}{2}\right)^{t_{A} +t_{B}-D_A}<1.
\end{equation*}

\noindent Thus, by this and equations (\ref{vdelta2}) we have
\begin{equation*}
\sum_{\mathcal{C}_{\mathbold{\lambda_{M_\delta}}} \in \mathcal{V}_\delta^{(B)}} A_{\mathbold{\lambda_{M_\delta}}}^{t_{A}}B_{\mathbold{\lambda_{M_\delta}}}^{D_{A}-t_{A}} \leq \left(\frac{1}{2}\right)^{t_{A} +t_{B}-D_A} \!\!\!\! \!\!\!\! \sum_{\mathcal{C}_{\mathbold{\lambda_{M_\delta}}} \in \mathcal{V}_\delta^{(B)}}B_{\mathbold{\lambda_{M_\delta}}}^{t_{B}}A_{\mathbold{\lambda_{M_\delta}}}^{D_{B}-t_{B}} \leq \left(\frac{1}{2}\right)^{t_{A} +t_{B}-D_A},
\end{equation*}
that again by (\ref{vdelta2}) imply 
\begin{equation}  \label{vdelta3}
\sum_{\mathcal{C}_{\mathbold{\lambda_{M_\delta}}}\in \mathcal{V}_\delta^{(A)}} A_{\mathbold{\lambda_{M_\delta}}}^{t_{A}}B_{\mathbold{\lambda_{M_\delta}}}^{D_{A}-t_{A}} \geq 1- \left(\frac{1}{2}\right)^{t_{A} +t_{B}-D_A}.
\end{equation} 

\noindent If we now define in $ D^{\mathbb{N}}$ the Bernouilli measure $\mu$  given by the vector of weights  
\begin{equation*}
p_{i j}=a_{i}^{t_{A}}b_{j}^{D_{A}-t_{A}},
\end{equation*}
we have by Remark \ref{remcover} and equation (\ref{vdelta3}) 
\begin{equation}\label{measureVd}
\mu (\mathcal{V}_\delta^{(A)})= \sum_{\mathcal{C}_{\mathbold{\lambda_{M_\delta}} }\in \mathcal{V}_\delta^{(A)}}  \mu ( \mathcal{C}_{\mathbold{\lambda_{M_\delta}} })   =  \sum_{\mathcal{C}_{\mathbold{\lambda_{M_\delta}} }\in \mathcal{V}_\delta^{(A)}}    A_{\mathbold{\lambda_{M_\delta}}}^{t_{A}}B_{\mathbold{\lambda_{M_\delta}}}^{D_{A}-t_{A}} \geq 1-\left(\frac{1}{2}\right)^{t_{A} +t_{B}-D_A}.\\
\end{equation}
\bigskip

\begin{lemma} \label{ngeqm} If $t_A+t_B > D_A \geq D_B$ then  
$m_k \leq n_k$ for $k$ large enough.
\end{lemma}
\begin{proof}
The definition of $m_k$ and $n_k$ in (\ref{mn}) imply
\begin{equation*}
\frac{\log m_k}{k}= - \! \! \!\sum_{(i,j)\in D}  \! \! \!p_{ij} \log a_{i}  + o(1) \qquad \frac{\log n_k}{k}= -\! \! \!\sum_{(i,j)\in D} \! \! \! p_{ij} \log b_j  + o(1).\\
\end{equation*} 

\noindent Let us consider the functions $f, g: D ^{\mathbb{N}} \rightarrow \mathbb{R} $ given by $f(\mathbold{\lambda})= -\log a_{\lambda_{0}}$, $g(\mathbold{\lambda})= -\log b_{\lambda_{0}}$, where $ (a_{\lambda_{0}},b_{\lambda_{0}})=(a_i,b_j)$ for $(i,j)=\lambda_0$. Note that both functions belong to $L^{1}({D}^{\mathbb{N}},\mathcal{B}, \mu)$, and since the shift map $\sigma $ in ${D}^{\mathbb{N}}$ is ergodic with respect to Bernoulli measures, we can  apply Birkhoff’s Ergodic Theorem to get 
\begin{equation} \label{birk1}
\begin{split}
- \lim_{n \rightarrow \infty}\dfrac{\log A_{\mathbold{\lambda_{n-1}}}}{n}&=-\lim_{n \rightarrow \infty}\dfrac{ \sum_{j=0}^{n-1} \log a_{\lambda_{j}}}{n }=\lim_{n\rightarrow \infty } \dfrac{1}{n} \sum_{j=0}^{n-1}f(\sigma^{j}(\mathbold{\lambda_{n-1}})) 
=\int_{D^{\mathbb{N}}} -\log a_{\lambda_{0}} d\mu(\mathbold{\lambda})\\
&=-\sum_{(i,j)\in D} \! p_{ij} \log a_{i} = \frac{\log m_k}{k}+ o(1)
\end{split}
\end{equation}
for $\mu$-almost all $\mathbold{\lambda} \in D^{\mathbb{N}}$.
And similarly 
\begin{equation}\label{birk2}
- \lim_{n \rightarrow \infty}\dfrac{\log B_{\mathbold{\lambda_{n-1}}}}{n}=\frac{\log n_k}{k}+ o(1)
\end{equation}
for $\mu$-almost all $\mathbold{\lambda} \in D^{\mathbb{N}}$.\\

Let $X\subset D^{\mathbb{N}}$ be the set of full measure for which equations (\ref{birk1}) and (\ref{birk2}) hold. For the seek of contradiction let us assume that $m_k>n_k$, i.e 
\begin{equation}\label{limitsB}
  \lim_{n \rightarrow \infty}\dfrac{\log (2A_{\mathbold{\lambda_{n-1}}})}{n}=\lim_{n \rightarrow \infty}\dfrac{\log A_{\mathbold{\lambda_{n-1}}}}{n} < \lim_{n \rightarrow \infty}\dfrac{\log B_{\mathbold{\lambda_{n-1}}}}{n} 
\end{equation}
for all $\mathbold{\lambda} \in X$ and $k$ large enough. Observe that by definition, $M_\delta(\mathbold{\lambda})$ converges to infinity when $\delta$ tends to $0$. Let $\delta_k=1/k$. Then, by equation (\ref{limitsB}), for each $\mathbold{\lambda}\in X$ there exists $\delta_{k}(\mathbold{\lambda})$ such that   $\mathcal{C}_{\mathbold{\lambda_{M_{\delta}}}} \in \mathcal{V}_{\delta}^{(B)}$ for all $\delta \leq \delta_{k}(\mathbold{\lambda})$. This means that the characteristic functions 
\[
 \chi_{\mathcal{V}_{\delta_k}^{(B)}}(\mathbold{\lambda_{M_{\delta_k}}}) =
  \begin{cases} 
      \hfill 1    \hfill & \text{ if $\mathcal{C}_{\mathbold{\lambda_{M_{\delta_k}}}} \in \mathcal{V}_{\delta_k}^{(B)}$} \\
      \hfill 0    \hfill & \text{ otherwise} \\
  \end{cases}
\]
converge pointwise to the constant function $1$ in $X$. Thus, by Egorov's Theorem, for all $\epsilon>0$ there is a measurable set $\tilde{X} \subset X$ with
$\mu(\tilde{X}) > 1- \epsilon$ and such that $\chi_{\mathcal{V}_{\delta_k}^{(B)}}$ converges uniformly on $\tilde{X}$, which contradicts (\ref{measureVd}). \\
\end{proof}

\noindent Let us define 
\begin{equation*}
s_{k}:=t_A + \frac{\log \vert \Gamma
_{k} \vert -t_A\log m_k }{ \log n_k}.
\end{equation*} 
\vspace{1pt}
\begin{lemma}\label{sktoDB}
For every $\underline{t} \in A$ there exists a sequence of Bedford-McMullen-type systems $\mathcal{I}_{k}$ with uniform fibers and attractors $\Lambda_k \subset F_{\underline{t}}$ for which
$$s_k \longrightarrow D_A$$
as $k$ tends to infinity.
\end{lemma}

\begin{proof}
Recall that by Stirling's formula we got equation (\ref{lim11}):
\begin{equation*}
\lim_{k\rightarrow \infty} \dfrac{\log |\Gamma_k |}{k}=- \! \! \!\sum_{(i,j)\in D} p_{i j} \log p_{i j}, 
\end{equation*}
and by the definition of $m_k$ and $n_k$,  
\begin{equation*}
\log m_k= - k \! \! \!\sum_{(i,j) \in D}  \! \! \!p_{ij}\log a_{i}  + o(k) \qquad \log n_k= -k \! \! \! \sum_{(i,j) \in D}  \! \! \!p_{ij}\log  b_j  + o(k)
\end{equation*}
Our choice of the vector $\textbf{p}$ implies $p_{ij}/a_i^{t_A}=b_j^{D_A-t_A}$ for all $(i,j) \in D$. Thus, 
\begin{equation*}
\lim_{k\rightarrow \infty}s_k = t_A+ \dfrac{\sum_{(i,j) \in D}  p_{ij}\log \left( \frac{p_{ij}}{a_{i}^{t_A}}\right) }{\sum_{(i,j) \in D} p_{ij}\log b_{j}}=t_A+  \frac{(D_A-t_A) \sum_{(i,j) \in D}  p_{ij}\log b_{j}}{\sum_{(i,j) \in D} p_{ij}\log b_{j}}=D_A
\end{equation*}
\end{proof}

\begin{rem} We emphasize that we are not claiming $s_k$ to be the box-counting dimension of the approximating carpets $\Lambda_k$. In fact, when the system is free of overlapping rows, this would already conflict with Fraser-Shmerkin's Theorem \ref{FSthm}.\\

Moreover, due to a possible drop on projected dimensions, the box dimension of the approximation systems might be smaller than the target one. Therefore, unlike in Section \ref{sectionHaus}, we won't be able to approximate by a system without rows overlapping and conclude by applying Fraser-Shmerkin's Theorem. Instead, we will recreate the argument on \cite[Section 6]{FS} to get subsystems satisfying the OSC, and then, instead of computing the dimensions of their associated attractors, we will consider in equation (\ref{lastsets}) images of our original attractor $F_{\underline{t}}$ under these maps, guaranteeing then the maximal projection dimension.\\
\end{rem}

Let $E$ be the set defined in (\ref{setE}). Then for any translation vector outside this set we can approximate the corresponding system by a subsystem satisfying the OSC and with ``enough maps''. The idea of the proof is the following: we apply Lemma \ref{G21} to the systems $\mathcal{I}_{k}$ in order to get new systems $\mathcal{L}_{\ell}$ which can only possibly have columns overlapping.  Then we approximate to a system $\mathcal{I}_{u}$ with uniform fibres which will be projected onto the horizontal axis in order to apply Hochman's results and Lemma \ref{vitali}. We will look at the new 1-dimensional system satisfying the SSC and consider the maps of the iteration of $\mathcal{I}_{u}$ whose projection belongs to it. Such system will satisfy the OSC and we will get a lower bound for its number of maps:

\begin{lemma} \label{lemmaB} Let $\underline{t} \in A \setminus E$ and $\mathcal{I}_{k}$ be the Bedford-McMullen-type system with uniform fibres defined in equation (\ref{dbcarpet2}). Then given $\epsilon > 0$ there exists $q_{0} \in \mathbb{N}$  such that for all $q \geq q_{0}$ we can define a new system $\mathcal{Q}_{q}=\lbrace S_{\underline{t},\lambda} \rbrace_{\lambda \in U_{q}}$ with $U_{q}\subseteq D^{\theta(k)q}$ that satisfies the OSC and so that
\begin{equation*}
\vert U_{q} \vert \geq 9^{-1}(m_{k} n_k)^{-q \epsilon} \vert \Gamma_{k}\vert^{q}e^{-\epsilon q \log n_k}.
\end{equation*}
\end{lemma}

\begin{proof}
Firstly, since the choice of the probability vector that defines $\Gamma_k$ in equation (\ref{BMcarpet}) doesn't play any role in the proof of Lemma \ref{G21}, we start by applying such lemma to our system $\mathcal{I}_{k}$ in order to get, for any $\ell \geq \ell_{0} \in \mathbb{N}$ a new system $\mathcal{L}_{\ell}=\lbrace S_{j} \rbrace_{j \in G_{k,\ell}}$ such that $\mathcal{L}_{\ell}^{Y}$ satisfies the OSC. With regard to the new number of maps we have
\begin{equation}\label{ecl2}
\vert G_{k, \ell} \vert \geq 
3^{-1}(1/n_k)^{ \ell \epsilon} \vert \Gamma_{k}\vert^{\ell} .
\end{equation} 
 
\noindent Now we approximate the system $\mathcal{L}_{\ell}$ by a system with uniform fibers: for $u\in\mathbb{N}$, set $\theta(u) = \sum_{\lambda \in G_{k, \ell}} \lceil pu \rceil=N \lceil u/N \rceil$, where $N=\vert G_{k, \ell} \vert$ and $p=1/N$. Note that 
\begin{equation} \label{k-n}
u - N \leq  \theta(u) \leq u  \qquad \text{  for all } u \in \mathbb{N},
\end{equation}
since $\theta(u)=N(\frac{u}{N}-\lbrace \frac{u}{N}\rbrace) \geq N(\frac{u}{N}-1)=u-N$, and the second inequality is trivial. Let us consider the set $ G_{k, \ell}^{\theta(u)}$ and define
\begin{equation*}
 H_{u}= \Bigg\{
 \begin{tabular}{c}
 $\mathbold{\lambda_{u}}=(\lambda_{1},\lambda_{2},\ldots ,\lambda_{\theta(u)})\in G_{k, \ell}^{\theta(u)} : \text{ for all } \lambda \in G_{k, \ell},$ \\
 $ |\lbrace n \in \lbrace 1, \ldots,\theta(u) \rbrace : \lambda_{n}=\lambda \rbrace|=  \lceil pu \rceil  $
 \end{tabular}
 \Bigg \}.
\end{equation*} 

\noindent By equation (\ref{k-n}) we can assume that $\theta (u)=u$. Note that reasoning as when obtaining equation (\ref{cardinals}), the cardinal of the set $H_{u}$ is given by
\begin{equation*} \label{Hu}
|H_{u}|=\dfrac{(N \lfloor u/ N \rfloor)!}{(\lfloor u/N \rfloor!)^{N}}.  
\end{equation*}
As in Lemma \ref{sktogp}, we can estimate its size when $u$ tends to infinity using Stirling's formula, that provided the equality (\ref{lim11}). The substitution $p_{ij}=1/N$ in such formula yields to 
\begin{equation*}\label{lemahu2}
\dfrac{\log|H_{u}|}{u}\rightarrow \log \vert G_{k, \ell} \vert. 
\end{equation*}

\noindent Thus, given $\epsilon >0$, there exists $u_0$ such that for all $u \geq u_0$
\begin{equation*}
\vert H_{u} \vert \geq \vert G_{k, \ell} \vert^{u}e^{-\epsilon u}. 
\end{equation*}\\ 
Let $\mathcal{H}_u:=\big \{ S_{\underline{t},\mathbold{ \lambda_{u}}}\big \}_{\mathbold{\lambda_{u}} \in H_u}$
be the IFS generated by $H_u$, and let $\mathcal{H}^{X}_u$ be its projected system into the horizontal axis, with attractor $\Psi_u^{X}$. Since  $\underline{t} \notin E$, by Theorem \ref{H1} we have that 
\begin{equation*} 
\dim_{B}(\Psi_u^{X})=\frac{\log \vert \overline{H}^{X}_u\vert}{u\ell\log m_k}=:\overline{s}_u^{X},
\end{equation*}
that satisfies $0\leq \overline{s}_u^{X}\leq 1$. 
Using Lemma \ref{vitali}, we can approximate the one-dimensional overlapping self-similar system $\mathcal{I}^{X}_{u}$ by a subsystem satisfying the SSC by assigning to the parameters of the mentioned lemma the values $\alpha= \overline{s}^{X}_{u}$, $a=m_k^{-u\ell }$. In particular there exists $v_{0} \in \mathbb{N}$ so that for $v \geq v_{0} $ we may find  
\[ \overline{U}^{X}_{v} \subset \overline{H}^{X}_u\]
such that the system $ \lbrace \overline{S}_{\underline{t},i}\rbrace_{i \in \overline{U}^{X}_{v}} $  satisfies the SSC, and
\begin{equation*}
\vert \overline{U}^{X}_{v} \vert \geq 3^{-\overline{s}_u^{X}}(1/m_k)^{- vu\ell(\overline{s}_u^{X}-\epsilon)}=3^{-\overline{s}^{X}_{u}}(1/m_k)^{vu\ell \epsilon} \vert \overline{H}^{X}_{u}\vert^{v}
\end{equation*}
since by equation (\ref{alphak}) we have $\vert \overline{H}^{X}_{u}\vert = m_k^{u \ell \overline{s}_u^{X}}$.\\

\noindent We fix any such $v \geq v_0$ and define the set  
\[U_{v}:=\lbrace ((i_{1},j_{1}),\ldots ,(i_{\theta(k)\ell uv},j_{\theta(k)\ell uv}) \in H_u^{v}:(j_{1},\ldots ,j_{\theta(k)\ell uv}))  \in \overline{U}^{X}_v \rbrace.\]
Let $\mathcal{U}_v=\lbrace S_{j} \rbrace_{j \in U_{v}}$, and observe that the system $\mathcal{H}_u$ having uniform fibers means that
\begin{equation} \label{ecl1}
\vert U_{v} \vert =\vert \overline{U}^{X}_{v}\vert\left(\frac{\vert H_{u} \vert }{\vert \overline{H}^{X}_{u} \vert }\right)^{v} \geq 3^{-\overline{s}^{X}_{u}}(1/m_{k})^{vu\ell \epsilon} \vert H_{u} \vert^{v}.
\end{equation} 
Thus, by equations (\ref{ecl2}), (\ref{ecl1}) and the fact that $\overline{s}^{X}_{u} \leq 1$ we get 
\begin{equation*} \label{G222}
\vert U_{v} \vert \geq 3^{-\overline{s}^{X}_{u}}(1/m_{k})^{vu\ell \epsilon}\left( \vert G_{k, \ell} \vert^{u}e^{-\epsilon u}\right)^{v} \geq 9^{-1}(m_{k} n_k)^{-vu \ell \epsilon} \vert \Gamma_{k}\vert^{vu\ell}e^{-\epsilon vu\ell \log n_k}.
\end{equation*}
The definitions $q_0:=v_0 u_0 \ell_0$ and $q:=v u \ell$ lead to the estimate of the statement.
\end{proof}
\vspace{8pt} 

\noindent We now prove the main result of this subsection, following a similar argument to that in \cite[Proof of Theorem 2.2]{FS}.

\begin{prop}\label{lowerB}
Let $\underline{t} \in A \setminus E$. Then
 \[ \dim_{B}(F_{\underline{t}})\geq \max (D_{A}, D_{B}),\]
where $D_{A}$, $D_{B}$ are the unique real numbers given by equation (\ref{DADB}).
\end{prop}

\begin{proof} Assume $t_A+t_B>D_A\geq D_B$, fix any $\epsilon >0$ and let $\mathcal{I}_k$ be the IFS defined in (\ref{dbcarpet2}), to which we apply Lemma \ref{lemmaB} to get the system $\mathcal{Q}_{q}$  satisfying the OSC defined by a set $U_q \subseteq D^{\theta(k)q}$. Lemma \ref{sktoDB} allow us, for a given $\epsilon >0$, to find $k\in \mathbb{N}$ such that $s_{k}\geq D_{A}-\epsilon$. By definition of $s_k$ it holds
\begin{equation}\label{mio}
\begin{aligned}
\left(\vert \Gamma_{k} \vert m_k^{-t_A} n_k^{-(s_{k}-t_A)} \right)=1.
\end{aligned}
\end{equation}
Let $r = (1/n_k)^{q}$. Observe that 
\begin{equation} \label{mio2}
e^{-\epsilon q \log n_k}=r^{-\epsilon},
\end{equation}
and consider the set
\begin{equation}\label{lastsets}
F_{0} := \bigcup_{\lambda \in U
_q} S_{\underline{t},\lambda}(F_{\underline{t}}) \subset F_{\underline{t}}. 
\end{equation}

We are going to get a lower bound for the dimension of $F_{\underline{t}}$ using the $\rho$-grid definition of $N_{\rho}(\cdot)$. By Proposition \ref{propbox}, there exists a constant $C_{\epsilon} > 0$ depending only on $\epsilon$ such that for all $\rho \in (0, 1]$ we have $N_{\rho}(\pi_{\text{\tiny X}} F_{\underline{t}}) \geq C_{\epsilon} \rho^{-(\overline{s}-\epsilon)},$ where $\overline{s} = \dim_{B}(\pi_{\text{\tiny X}} F_{\underline{t}})$. In our case, as $\underline{t} \in A\setminus E$, by Theorem \ref{H1} $\dim_{B}(\pi_{\text{\tiny X}} F_{\underline{t}})=t_A$. Thus, choosing $\rho=rm_k^{q}$ we get
\begin{equation*} \label{rm}
N_{rm_k^{q}}(\pi_{\text{\tiny X}} F_{\underline{t}}) \geq C_{\epsilon}\left( \frac{(1/m_k)^{q}}{r}\right)^{t_A-\epsilon}.
\end{equation*}\\
Note that each set $S_{\underline{t},\lambda}(F_{\underline{t}})$ in the composition of $F_{0}$ is contained in the rectangle $S_{\underline{t},\lambda}([0,1]^{2})$ which has height $r$ and base length $(1/m_k)^{q}$. Covering a set $S_{\underline{t},\lambda} (F_{\underline{t}})$ by squares of size $r$ is equivalent to covering $\pi_{\text{\tiny X}} (F_{\underline{t}})$ by intervals of length $rm_k^{q}$. It follows that
\begin{equation}\label{casi}
N_{r}(S_{\underline{t},\lambda} (F_{\underline{t}})) = N_{rm_k^{q}}( \pi_{\text{\tiny X}} (F_{\underline{t}})) \geq C_{\epsilon}\left( \dfrac{(1/m_k)^{q}}{r}\right)^{t_A-\epsilon}.
\end{equation}

Let $U$ be any closed square of side-length $r$. Since by Lemma \ref{lemmaB} $\lbrace S_{\underline{t},\lambda}([0,1]^{2})\rbrace_{\lambda \in U_q}$ is a collection of rectangles which can only intersect at the boundaries, each rectangle with shortest side of length $r$ by Lemma \ref{ngeqm}, our square $U$ can intersect no more than $9$ of the sets $\lbrace S_{\lambda}(F_{\underline{t}})\rbrace_{\lambda \in U_q}$. Thus, by (\ref{lastsets}) we have 
\[\sum_{\lambda \in U_q} N_{r}(S_{\underline{t},\lambda} (F_{\underline{t}})) \leq 9 \: N_{r} \bigg( \bigcup_{\lambda \in U_q} S_{\underline{t},\lambda}(F_{\underline{t}}) \bigg) \leq 9\: N_{r}(F_{\underline{t}}). \]

\noindent Equations (\ref{casi}), (\ref{mio2}), Lemma  \ref{lemmaB} and (\ref{mio}) successively imply the chain of inequalities
\begin{equation*}
N_{r}(F_{\underline{t}})  \geq \dfrac{1}{9} \sum_{\lambda \in U_q} N_{r}(S_{\underline{t},\lambda} (F_{\underline{t}})) \geq \dfrac{1}{9} \vert U_q \vert C_{\epsilon} \left( \dfrac{(1/m_k)^{q}}{r}\right)^{(t_A-\epsilon)}\geq  \dfrac{C_{\epsilon}}{81}\vert \Gamma_{k}\vert^{q} m_k^{-q t_A} r^{(\epsilon-t_A)} \geq \dfrac{C_{\epsilon}}{81}r^{-(s_k-\epsilon)}.
\end{equation*}

\noindent This is valid for all $q \geq q_{0}$, and hence
\[\liminf_{q \rightarrow \infty}\frac{\log N_{(1/n_k)^{q}}(F_{\underline{t}})}{-\log (1/n_k)^{q}}\geq s_{k}-\epsilon \geq D_{A} -2\epsilon \] 
 
\noindent by the choice of $ {s}_{k} $.  By equation (\ref{eqBox}) in Proposition \ref{propbox}, letting $r$ tend to zero through the sequence $(1/n_k)^{q}$ as $ q \rightarrow \infty$ is sufficient to give a lower bound on the lower box dimension of $F_{\underline{t}}$. Since $ \epsilon $ can be made arbitrarily small, this yields $\underline{\dim}_{B}F_{\underline{t}} \geq D_{A}$ as required.\\

We are left to deal with the case when $t_A+t_B =D_A$ (and therefore $t_A+t_B =D_A=D_B$). In this case both equations in (\ref{DADB}) become $\sum_{(i,j)\in D} a_{i}^{t_{A}}b_{j}^{t_{B}}=1$. Therefore, by equation (\ref{minus1}), it must occur for all $i \in \overline{D}_{X}$ that $I_i=\lbrace (i,j): j\in  \overline{D}_{Y} \rbrace$; in other words, the contractions that define our fixed Bara\'{n}ski system map the unit square to a ``full grid'' of $\vert \overline{D}_{X}\vert \times \vert\overline{D}_{Y}\vert$ rectangles. In particular, its attractor can be expressed as $F=\pi_X(F)\times \pi_Y(F)$. The same is still true when we allow overlaps induced by any translating parameters $\underline{t} \in A$, and so $F_{\underline{t}}=\pi_X(F_{\underline{t}})\times \pi_Y(F_{\underline{t}})$. Hence, using the same argument as in Lemma \ref{dimE}, and by Proposition \ref{H2}, if $\underline{t} \in A \setminus E$ we have that $$\dim_B(F_{\underline{t}})=\dim_B(\pi_X(F_{\underline{t}})\times \pi_Y(F_{\underline{t}}))\geq t_A+t_B=D_A \geq\max(D_A,D_B).$$ \end{proof}

\subsection{Calculation of the dimension}
\begin{proof}[Proof of Theorem \ref{THMbox}]
It follows directly from Propositions \ref{upperB}, \ref{lowerB} and Lemma \ref{dimE}.
\end{proof}

\begin{proof}[Proof of Corollary \ref{corBox}]
A Bedford-McMullen-type system of parameters $(\tilde{m}, \tilde{n})$ is in particular Bara\'{n}ski system, and thus Theorem \ref{THMbox} applies. Therefore, outside a set of parameters $E$, equations (\ref{DADB}) and (\ref{TB}) hold, that in this case become
\begin{equation*}
\vert D \vert (1/\tilde{m})^{t_{A}}(1/\tilde{n})^{D_{A}-t_{A}}=1, \qquad \qquad \vert D \vert (1/\tilde{n})^{t_{B}}(1/\tilde{m})^{D_{B}-t_{B}}=1, 
\end{equation*}
where $t_{A}$ and $t_{B}$ are given by
\begin{equation*}
\vert \overline{D}_{X}\vert (1/\tilde{m})^{t_{A}}=1, \qquad \qquad \vert \overline{D}_{Y}\vert (1/\tilde{n})^{t_{B}}=1, 
\end{equation*}
and therefore 

\begin{equation}\label{pruebaB}
D_A= \frac{\log \vert \overline{D}_{X}\vert }{\log \tilde{m}}+\dfrac{\log (\vert D\vert/\vert \overline{D}_{X}\vert )}{\log \tilde{n}} \qquad \qquad D_B= \frac{\log \vert \overline{D}_{Y}\vert }{\log \tilde{n}}+\dfrac{\log (\vert D\vert/\vert \overline{D}_{Y}\vert )}{\log \tilde{m}}.
\end{equation}\\
\vspace{-2pt}

Recall that in the definition of Bedford-McMullen-type carpet we have assumed $\tilde{n}>\tilde{m} >1$, and it is always true that $\vert D \vert \leq \vert \overline{D}_{X} \vert \vert\overline{D}_{Y}\vert$. Thus,
\begin{equation*}
\tilde{n}\geq\tilde{m} \Leftrightarrow \frac{\log \left( \frac{\vert \overline{D}_{X} \vert \vert\overline{D}_{Y}\vert}{\vert D \vert} \right) }{\log \tilde{m}} \geq   \frac{\log \left( \frac{\vert \overline{D}_{X} \vert \vert\overline{D}_{Y}\vert}{\vert D \vert} \right) }{\log \tilde{n}} \Leftrightarrow D_A \geq D_B
\end{equation*}
by equation (\ref{pruebaB}). We can assume that there exists at least one column  $i \in \vert \overline{D}_{X} \vert $ with more than one element in $D$ (otherwise there is no possible overlap and $\dim_B F_{\underline{t}}= \vert D \vert / \log \tilde{m}$ for any $\underline{t}$), that is, there exist $i$, $j_{1}$ and $j_{2}$ such that  $\lbrace(i,j_{1}),(i,j_{2})\rbrace \subset D $. Let us consider the hyperplane $$\mathcal{P}=\lbrace \underline{t} \in A : t_{j_{1}}=t_{j_{2}}  \rbrace.$$ 

This merges two rectangles of our original pattern, so that $\vert D \vert$ decreases in at least one without increasing the total number of columns $\vert \overline{D}_{X} \vert$. Therefore, for any $\underline{t} \in \mathcal{P}$, by Proposition \ref{upperB} and equation (\ref{pruebaB})
$$  \overline{\dim}_{B}(F_{\underline{t}})\leq \frac{\log \vert \overline{D}_{X}\vert }{\log \tilde{m}}+\dfrac{\log ((\vert D\vert-1)/\vert \overline{D}_{X}\vert )}{\log \tilde{n}} <D_A.$$
\vspace{1pt}

\noindent Thus, $\mathcal{P} \subseteq E_1$, and since $\dim_{H}\mathcal{P}=\vert\overline{D}_{X}\vert +  \vert \overline{D}_{Y}\vert-1$, we have $\dim_{H} E_1 \geq \vert\overline{D}_{X}\vert + \vert  \overline{D}_{Y}\vert-1$. Note that $E_1\subseteq E$, although the sets are not necessarily equal. But this inclusion implies that $\dim_{H}E_1\leq \dim_{H}E= \vert \overline{D}_{Y}\vert + \vert \overline{D}_{Y}\vert-1$, by Lemma \ref{dimE}$(b)$. Hence, by the sandwich lemma $\dim_{H}E_1=\vert\overline{D}_{X}\vert + \vert \overline{D}_{Y}\vert-1$. The same argument applies to the packing dimension of $E_1$, which concludes the proof.  
\end{proof}

\bibliographystyle{alpha}
\bibliography{biblio}
\end{document}